\documentclass[9pt]{article}
\usepackage{amsfonts}
\usepackage{amssymb,amsmath,euscript,bbm,xcolor,graphicx,epstopdf}
\usepackage[titletoc]{appendix}

\newtheorem{theorem}{Theorem}

\newtheorem{condition}[theorem]{Condition}

\newtheorem{lemma}[theorem]{Lemma}

\newtheorem{proposition}[theorem]{Proposition}
\newtheorem{remark}[theorem]{Remark}

\newenvironment{proof}[1][Proof]{\noindent\textbf{#1.} }{\ \rule{0.5em}{0.5em}}
\input{tcilatex}
\begin{document}

\title{A Quantitative Central Limit Theorem for the Euler-Poincar\'e
Characteristic of Random Spherical Eigenfunctions}
\author{Valentina Cammarota and Domenico Marinucci \\
%EndAName
Department of Mathematics, University of Rome Tor Vergata}
\maketitle

\begin{abstract}
We establish here a Quantitative Central Limit Theorem (in Wasserstein
distance) for the Euler-Poincar\'{e} Characteristic of excursion sets of
random spherical eigenfunctions in dimension 2. Our proof is based upon a
decomposition of the Euler-Poincar\'{e} Characteristic into different
Wiener-chaos components: we prove that its asymptotic behaviour is dominated
by a single term, corresponding to the chaotic component of order two. As a
consequence, we show how the asymptotic dependence on the threshold level $u$
is fully degenerate, i.e., the Euler-Poincar\'{e} Characteristic converges
to a single random variable times a deterministic function of the threshold.
This deterministic function has a zero at the origin, where the variance is
thus asymptotically of smaller order. Our results can be written as an
asymptotic second-order Gaussian Kinematic Formula for the excursion sets of
Gaussian spherical harmonics.
\end{abstract}

\begin{itemize}
\item \textbf{AMS Classification}: 60G60, 62M15, 53C65, 42C10, 33C55.

\item \textbf{Keywords and Phrases}: Euler-Poincar\'{e} characteristic,
Wiener-Chaos Expansion, Spherical Harmonics, Quantitative Central Limit
Theorem, Gaussian Kinematic Formula, Berry's Cancellation Phenomenon
\end{itemize}

\section{Introduction}

The Euler-Poincar\'{e} Characteristic is perhaps the single most important
tool for the analysis of excursion sets for Gaussian random fields;
classical textbooks on its behaviour are \cite{adlertaylor}, \cite%
{adlerstflour}, while some very recent contributions can be found for
instance in \cite{taylorvadlamani}, \cite{estradeleon, chengxiaoa}, \cite%
{chengxiaob}, \cite{MV}. As well-known the Euler-Poincar\'{e}
Characteristic, which we shall denote by $\chi (\cdot )$, is the unique
integer-valued functional, defined on the ring $\mathcal{C}$ of closed
convex sets in $\mathbb{R}^{N}$, such that $\chi (A)=0$ if $A=\emptyset $, $%
\chi (A)=1$ if $A$ is homotopic to the unit ball, and which satisfies the
additivity property 
\begin{equation*}
\chi (A\cup B)=\chi (A)+\chi (B)-\chi (A\cap B),\hspace{1cm}\text{for all\ }%
A,B\in \mathcal{C}.
\end{equation*}%
The investigation of its behaviour for the excursion sets of Gaussian random
fields has now a rather long history: seminal contributions were given by
Robert Adler and his co-authors in the seventies; the area was then very
much revived by the discovery of the beautiful Gaussian Kinematic Formula 
\cite{taylor, adlertaylor}.

More precisely, let us denote by $f$ a real valued random field defined on
some manifold $\mathbb{M}$; as usual the excursion sets are defined by, for $%
u\in {\mathbb{R}},$ 
\begin{equation*}
A_{u}(f;\mathbb{M})=\left\{ x\in \mathbb{M}:f(x)\geq u\right\} \text{ .}%
\hspace{1cm}
\end{equation*}%
We write $\mathcal{L}_{j}^{f}$, $j=0,\dots ,\text{dim}(\mathbb{M})$, for the 
\emph{Lipschitz-Killing curvatures} (also known as \emph{intrinsic volumes})
of the manifold $\mathbb{M}$ under the Riemannian metric $g^{f}$ induced by
the covariance of $f$; in other words, for $U_{x},V_{x}$ that belong to $%
T_{x}\mathbb{M}$, the tangent space to $\mathbb{M}$ at $x,$ we have 
\begin{equation}
g_{x}^{f}(U_{x},V_{x}):=\mathbb{E}[(U_{x}f)\cdot (V_{x}f)],
\label{Riemetric}
\end{equation}%
(see \cite{taylor},\cite{adlertaylor} for further details); in particular $%
\mathcal{L}_{0}$ is the Euler-Poincar\'{e} Characteristic. To introduce the
Gaussian Kinematic Formula, we need to consider also the functions $\rho
_{j} $, which are labelled Gaussian Minkowski functionals and defined by 
\begin{equation*}
{\rho }_{j}(u)=(2\pi )^{-(j+1)/2}H_{j-1}(u)e^{-u^{2}/2};
\end{equation*}%
here $H_{q}(\cdot )$ are the Hermite polynomial of order $q$, which satisfy
(see i.e., \cite{noupebook})%
\begin{equation*}
H_{-1}(u)e^{-u^{2}/2}:=1-\Phi (u),\hspace{1cm}H_{j}(u)=(-1)^{j}(\phi
(u))^{-1}\frac{d^{j}}{du^{j}}\phi (u),\hspace{0.4cm}j=0,1,\dots ,
\end{equation*}%
$\phi (\cdot )$, $\Phi (\cdot )$ denoting the standard Gaussian density and
distribution functions, respectively. For instance, the first few Hermite
polynomials are given by: 
\begin{equation*}
H_{0}(u)=1,\hspace{0.3cm}H_{1}(u)=u,\hspace{0.3cm}H_{2}(u)=u^{2}-1,\hspace{%
0.3cm}H_{3}(u)=u^{3}-3u,\dots
\end{equation*}%
For a smooth, centred, unit variance, Gaussian random fields $f:\mathbb{M}%
\rightarrow \mathbb{R}$ the Gaussian Kinematic Formula then implies that the
expected Euler-Poincar\'{e} Characteristic of the excursion sets is given by 
\begin{equation}
\mathbb{E}[\chi (A_{u}(f;\mathbb{M}))]=\sum_{j=0}^{\text{dim}(\mathbb{M})}{%
\mathcal{L}}_{j}^{f}(\mathbb{M}){\rho }_{j}(u).  \label{GKFA}
\end{equation}

More recently, a formula which can be viewed as an higher order extension of
the Gaussian Kinematic Formula for the covariance of the Euler-Poincar\'{e}
Characteristic characteristic of excursion sets at different thresholds, was
established by \cite{CMW-EPC}, who focussed on an important class of fields:
Gaussian spherical harmonics. Indeed, consider the Laplace equation 
\begin{equation*}
\Delta _{\mathbb{S}^{2}}f_{\ell }+\lambda _{\ell }f_{\ell }=0,\hspace{1cm}%
f_{\ell }:\mathbb{S}^{2}\rightarrow \mathbb{\mathbb{R}},
\end{equation*}%
where $\Delta _{\mathbb{S}^{2}}$ is the Laplace-Beltrami operator on the
unit sphere $\mathbb{S}^{2}$ and $\lambda _{\ell }=\ell (\ell +1)$, $\ell
=0,1,2,\dots $. For a given eigenvalue $-\lambda _{\ell }$, the
corresponding eigenspace is the $(2\ell +1)$-dimensional space of spherical
harmonics of degree $\ell $; we can choose an arbitrary $L^{2}$-orthonormal
basis $\left\{ Y_{\mathbb{\ell }m}(.)\right\} _{m=-\ell ,\dots ,\ell }$, and
consider random eigenfunctions of the form 
\begin{equation}
f_{\ell }(x)=\sqrt{\frac{4\pi }{2\ell +1}}\sum_{m=-\ell }^{\ell }a_{\ell
m}Y_{\ell m}(x),  \label{fell}
\end{equation}%
where the coefficients $\left\{ a_{\mathbb{\ell }m}\right\} $ are
complex-valued Gaussian variables, such that for $m\neq 0$, $\mathrm{Re}%
(a_{\ell m})$, $\mathrm{Im}(a_{\ell m})$ are zero-mean, independent Gaussian
variables with variance $\frac{1}{2}$, while $a_{\ell 0}$ follows a standard
Gaussian distribution with zero mean and unit variance; the law of the
process $\left\{ f_{\ell }(.)\right\} $ is invariant with respect to the
choice of a $L^{2}$-orthonormal basis $\{Y_{\ell m}\}$. Note that in this
paper we choose the basis of complex valued spherical harmonics instead of
the real ones that were adopted in \cite{CMW, CW}. Random spherical
harmonics arise naturally from Fourier analysis of isotropic spherical
random fields and in the investigation of quantum chaos, and they have hence
drawn quite a lot of interest in the last few years (see for instance \cite%
{canzani,DI,MR2015,nazarov,sarnakwigman,Wig,wigmanreview}); as discussed
below, we believe the results presented in this case can be extended to
Gaussian eigenfunctions on more general compact manifolds, but we leave this
issue for future research.

The random fields $\{f_{\ell }(x),\;x\in \mathbb{S}^{2}\}$ are centred,
Gaussian and isotropic, meaning that the probability laws of $f_{\ell
}(\cdot )$ and $f_{\ell }(g\cdot )$ are the same for any rotation $g\in
SO(3) $. From the addition theorem for spherical harmonics (\cite{MaPeCUP},
equation 3.42), the covariance function is given by 
\begin{equation*}
\mathbb{E}[f_{\ell }(x)f_{\ell }(y)]=P_{\ell }(\cos d(x,y)),
\end{equation*}%
where $P_{\ell }$ are the Legendre polynomials and $d(x,y)$ is the spherical
geodesic distance between $x$ and $y,$ i.e.%
\begin{equation*}
d(x,y)=\arccos (\left\langle x,y\right\rangle )\text{ . }
\end{equation*}%
An application of the Gaussian Kinematic Formula (\ref{GKFA}) gives in these
circumstances: 
\begin{equation}
\mathbb{E}[\chi (A_{u}(f_{\ell };\mathbb{S}^{2}))]=\frac{\sqrt{2}}{\sqrt{\pi 
}}\exp \{-u^{2}/2\}u\frac{\ell (\ell +1)}{2}+2[1-\Phi (u)],  \label{EPCs}
\end{equation}%
for a proof of formula (\ref{EPCs}) see, for example, \cite{MV}, Corollary
5, or \cite{chengxiaoa}, Lemma 3.5. In \cite{CMW-EPC}, the results on the
expected value were extended to an (asymptotic) evaluation of the variance;
in particular, it was shown that, as $\ell \rightarrow \infty $%
\begin{equation}
\mathrm{Var}[\chi (A_{u}(f_{\ell };\mathbb{S}^{2}))]=\frac{\ell ^{3}}{8\pi }%
(u^{3}-u)^{2}e^{-u^{2}}+O(\ell ^{2}\log ^{2}\ell ),  \label{variance0}
\end{equation}%
an expression that can be rewritten as 
\begin{equation*}
\ell \frac{\lambda _{\ell }}{4}\left\{ H_{1}(u)H_{2}(u)\phi (u)\right\}
^{2}+O(\ell ^{2}\log ^{2}\ell )
\end{equation*}%
or equivalently 
\begin{equation}
\ell \frac{\lambda _{\ell }}{4}\left\{ \left( H_{2}^{\prime
}(u)+H_{3}(u)\right) \phi (u)\right\} ^{2}+O(\ell ^{2}\log ^{2}\ell ),
\label{variance}
\end{equation}%
where $\phi (u)=\frac{1}{\sqrt{2\pi }}e^{-u^{2}/2}$ denotes as before the
standard Gaussian density function. This expression was derived by an
analytic computation, in turn a consequence of a rather hard analysis on the
asymptotic variance of critical points which was given in \cite{CMW, CW}.
Asymptotic expressions for the variances of the two other Lipschitz-Killing
curvatures for excursion sets in two dimensions, i.e. the area and (half)
the boundary length, were also given in \cite{DI}, \cite{MR2015},\cite%
{MW2014} and \cite{ROSSI2015},\cite{wigmanreview}; in \cite{CMW-EPC} all
these expressions were collected in a unitary framework and it was
conjectured that they could point out to a second-order extension of the
Gaussian Kinematic Formula for random eigenfunctions. A further contribution
in this direction is indeed given by our results in this paper, which we
present below.

\subsection{Main Results}

The main purpose of this paper is to show that the high frequency behaviour
is dominated (in the $L^{2}$ sense) by a single term with a very simple
analytic expression, whose variance is indeed given by (\ref{variance}). In
order to achieve this goal, we shall first establish the $L^{2}$ expansion
of $\chi (A_{u}(f_{\ell };\mathbb{S}^{2}))$ into Wiener chaoses (see (\ref%
{chaos}) below), which we will write as 
\begin{equation*}
\chi (A_{u}(f_{\ell };\mathbb{S}^{2}))-\mathbb{E}[\chi (A_{u}(f_{\ell };%
\mathbb{S}^{2}))]=\sum_{q=2}^{\infty }\mathtt{Proj}[\chi (A_{u}(f_{\ell };%
\mathbb{S}^{2}))|q]\text{ .}
\end{equation*}%
In the Euclidean case, a similar expansion was exploited in the recent paper 
\cite{estradeleon}; in our setting, however, the asymptotic behaviour of the
projection components turns out to be even neater; in particular, we shall
show that the projection onto the second-order chaos has the following, very
simple expression:

\begin{theorem}
\label{proj2} For all $\ell $ such that Condition \ref{Thecondition} holds,
we have 
\begin{align*}
\mathtt{Proj}[\chi (A_{u}(f_{\ell };\mathbb{S}^{2}))|2]& =\frac{\lambda
_{\ell }}{2}\left\{ H_{1}(u)H_{2}(u)\phi (u)\right\} \frac{1}{2\ell +1}%
\,\sum_{m=-\ell }^{\ell }\{|a_{\ell m}|^{2}-1\}+R(\ell ) \\
& =\frac{\lambda _{\ell }}{2}\left\{ H_{1}(u)H_{2}(u)\phi (u)\right\} \frac{1%
}{4\pi }\int_{\mathbb{S}^{2}}H_{2}(f_{\ell }(x))dx+R(\ell ),
\end{align*}%
where the remainder term $R(\ell )$ is such that $\mathbb{E}\left\vert
R(\ell )\right\vert ^{2}=O(\ell ^{2}\log \ell ),$ uniformly over $u.$\bigskip
\end{theorem}

Note that the variance of the first term on the right-hand side is equal to%
\begin{eqnarray*}
&&Var\left[ \frac{\lambda _{\ell }}{2}\left\{ H_{1}(u)H_{2}(u)\phi
(u)\right\} \frac{1}{2\ell +1}\,\sum_{m=-\ell }^{\ell }\{|a_{\ell m}|^{2}-1\}%
\right] \\
&=&\frac{\ell ^{2}(\ell +1)^{2}}{4}\phi ^{2}(u)(u^{3}-u)^{2}\frac{2}{2\ell +1%
}=\frac{\ell ^{3}}{8\pi }(u^{3}-u)^{2}e^{-u^{2}}+O(\ell ^{2})\text{ ,}
\end{eqnarray*}%
which is asymptotically equivalent to the variance of the Euler-Poincar\'{e}
Characteristic reported in (\ref{variance0}), so that the contribution from
all the remaining Wiener chaos terms is indeed of smaller order for every $%
u\neq 0.$ In view of this result, the investigation of the asymptotic
distribution becomes indeed much less difficult, and we can prove the second
main result of this paper, i.e.,

\begin{theorem}
\label{clt} \bigskip There exists a constant $K>0$ such that, for all $\ell $
fulfilling Condition \ref{Thecondition} and uniformly over $u\neq 0,$ we have%
\begin{equation*}
\mathbb{E}\left\{ \frac{\chi (A_{u}(f_{\ell };\mathbb{S}^{2}))-\mathbb{E}%
[\chi (A_{u}(f_{\ell };\mathbb{S}^{2}))]-\mathtt{Proj}[\chi (A_{u}(f_{\ell };%
\mathbb{S}^{2}))|2]}{\sqrt{\mathrm{Var}[\chi (A_{u}(f_{\ell };\mathbb{S}%
^{2}))]}}\right\} ^{2}\leq K\frac{\log \ell }{\ell },
\end{equation*}%
and 
\begin{equation*}
d_{W}\left( \frac{\chi (A_{u}(f_{\ell };\mathbb{S}^{2}))-\mathbb{E}[\chi
(A_{u}(f_{\ell };\mathbb{S}^{2}))]}{\sqrt{\mathrm{Var}[\chi (A_{u}(f_{\ell };%
\mathbb{S}^{2}))]}},Z\right) \leq K\sqrt{\frac{\log \ell }{\ell }},\text{ }
\end{equation*}%
$d_{W}(.,.)$ denoting as usual the Wasserstein distance and $Z\sim N(0,1)$ a
standard Gaussian variable.
\end{theorem}

We remark that the possibility to obtain simple, analytic formulae for the
second-order chaos component and its variance, together with sharp bounds on
the convergence in Wasserstein distance, are both peculiar features which do
not have analogous counterparts for the Euclidean domain results (see i.e., 
\cite{estradeleon}). Also, note that the asymptotic dependence on the
threshold level $u$ is fully degenerate, i.e. the Euler-Poincar\'{e}
Characteristic converges in mean square to a single random variable times a
deterministic function of the threshold, in the high-frequency limit $\ell
\rightarrow \infty $. All these features follow by the fact that a single
chaotic projection (the component of order 2) is dominating the asymptotic
behaviour of the Euler-Poincar\'{e} Characteristic; in the next subsection
we discuss this issue and cast into the more general framework of
Lipschitz-Killing curvatures for excursion sets of Gaussian eigenfunctions.

\subsection{Discussion}

\subsubsection{Some Recent Results on Lipschitz-Killing curvatures for
Gaussian Eigenfunctions}

The fact that the asymptotic behaviour of the Lipschitz-Killing curvatures
in the high frequency - high energy limit is dominated by the second-order
chaotic component, which disappears at level $u=0,$ seems to be of a general
nature when dealing with excursion sets of random eigenfunctions. The
simplest example of a Lipschitz-Killing curvature is given of course by the
excursion area; in this case, it was shown in \cite{DI} that 
\begin{align*}
\mathtt{Proj}[\mathcal{L}_{2}(A_{u}(f_{\ell };\mathbb{S}^{2}))|2]& =\frac{1}{%
2}u\phi (u)\int_{\mathbb{S}^{2}}H_{2}(f_{\ell }(x))dx \\
& =\frac{1}{2}u\phi (u)\frac{4\pi }{2\ell +1}\sum_{m=-\ell }^{\ell }\left\{
|a_{\ell m}|^{2}-1\right\}
\end{align*}%
and moreover, as $\ell \rightarrow \infty $, 
\begin{equation*}
\lim_{\ell \rightarrow \infty }\frac{\mathrm{Var}[\mathtt{Proj}[\mathcal{L}%
_{2}(A_{u}(f_{\ell };\mathbb{S}^{2}))|2]]}{\mathrm{Var}[\mathcal{L}%
_{2}(A_{u}(f_{\ell };\mathbb{S}^{2}))]}=O\left( \frac{1}{\ell }\right) ,
\end{equation*}%
\begin{equation*}
\lim_{\ell \rightarrow \infty }\mathbb{E}\left\{ \frac{\mathcal{L}%
_{2}(A_{u}(f_{\ell };\mathbb{S}^{2}))-\mathbb{E}[\mathcal{L}%
_{2}(A_{u}(f_{\ell };\mathbb{S}^{2}))]-\mathtt{Proj}[\mathcal{L}%
_{2}(A_{u}(f_{\ell };\mathbb{S}^{2}))|2]}{\sqrt{\mathrm{Var}[\mathcal{L}%
_{2}(A_{u}(f_{\ell };\mathbb{S}^{2}))]}}\right\} ^{2}=0.
\end{equation*}

\medskip This results were further investigated and extended to spheres of
arbitrary dimensions in \cite{MR2015}; again, easy consequences are

\begin{enumerate}
\item A Quantitative Central Limit Theorem in Wasserstein distance;

\item Asymptotic degeneracy of the multivariate distribution for different
thresholds $(u_{1},...,u_{p}),$ i.e., perfect correlation of the excursion
area at different thresholds

\item The fact that the variance at level $u=0$ is lower-order (related to
the so-called `Berry's cancellation phenomenon', see below).
\end{enumerate}

Another step in this literature was the analysis of the boundary length for $%
u=0$ for random eigenfunctions on the torus, led by \cite{MPRW2015}; i.e.,
the so-called nodal lines for arithmetic random waves, whose variance was
firstly established in \cite{KKW}. It should be noted that the nodal lines
for arithmetic random waves are indeed (twice) their Lipschitz-Killing
curvature of order 1 for $u=0,$ i.e. $\mathcal{L}_{1}(A_{0}(e_{k};\mathbb{T}%
^{2})),$ where we use $\mathbb{T}^{2}$ to denote the two-dimensional torus
and $e_{k}$ to denote its eigenfunctions, and $k$ is an integer such that $%
k^{2}=k_{1}^{2}+k_{2}^{2}$, for some $k_{1},k_{2}\in \mathbb{N}$. The
findings in \cite{MPRW2015} are indeed perfectly complementary to our
investigation here: it is shown that the behaviour of nodal lines is
dominated by a single term that corresponds to the fourth-order chaos
component, consistent with the vanishing of the second-order term when $u=0.$
Furthermore, in the (so far unpublished) Ph.D. thesis \cite{ROSSI2015} it is
shown that for the first Lipschitz-Killing curvature, i.e. half the length
of level curves of excursion sets of spherical eigenfunctions, one has also
(Proposition 7.3.1, page 116) 
\begin{eqnarray*}
\mathtt{Proj}[\mathcal{L}_{1}(A_{u}(f_{\ell };\mathbb{S}^{2}))|2] &=&\frac{1%
}{2}\sqrt{\frac{\ell (\ell +1)}{2}}\sqrt{\frac{\pi }{8}}u^{2}\phi (u)\int_{%
\mathbb{S}^{2}}H_{2}(f_{\ell }(x))dx \\
&=&\frac{1}{2}\sqrt{\frac{\ell (\ell +1)}{2}}\sqrt{\frac{\pi }{8}}u^{2}\phi
(u)\frac{4\pi }{2\ell +1}\sum_{m=-\ell }^{\ell }\left\{ |a_{\ell
m}|^{2}-1\right\} \text{ ,}
\end{eqnarray*}%
and thus again, as $\ell \rightarrow \infty $, 
\begin{equation*}
\lim_{\ell \rightarrow \infty }\frac{\mathrm{Var}[\mathtt{Proj}[\mathcal{L}%
_{1}(A_{u}(f_{\ell };\mathbb{S}^{2}))|2]]}{\mathrm{Var}[\mathcal{L}%
_{1}(A_{u}(f_{\ell };\mathbb{S}^{2}))]}=O\left( \frac{1}{\ell }\right) ,
\end{equation*}%
\begin{equation*}
\lim_{\ell \rightarrow \infty }\mathbb{E}\left\{ \frac{\mathcal{L}%
_{1}(A_{u}(f_{\ell };\mathbb{S}^{2}))-\mathbb{E}[\mathcal{L}%
_{1}(A_{u}(f_{\ell };\mathbb{S}^{2}))]-\mathtt{Proj}[\mathcal{L}%
_{1}(A_{u}(f_{\ell };\mathbb{S}^{2}))|2]}{\sqrt{\mathrm{Var}[\mathcal{L}%
_{1}(A_{u}(f_{\ell };\mathbb{S}^{2}))]}}\right\} ^{2}=0.
\end{equation*}

\subsubsection{\protect\bigskip A Second-Order Gaussian Kinematic Formula
for Random Spherical Harmonics}

The expressions we reported so far can be summarized into a single analytic
form as follows, for $k=0,1,2$, 
\begin{equation*}
\mathtt{Proj}[\mathcal{L}_{k}(A_{u}(f_{\ell };\mathbb{S}^{2}))|2]
\end{equation*}%
\begin{equation}
=\frac{1}{2}\left[ 
\begin{array}{c}
2 \\ 
k%
\end{array}%
\right] \left\{ \frac{\lambda _{\ell }}{2}\right\}
^{(2-k)/2}H_{1}(u)H_{2-k}(u)\phi (u)\frac{1}{(2\pi )^{(2-k)/2}}\int_{\mathbb{%
S}^{2}}H_{2}(f_{\ell }(x))dx+a_{k}(\ell ),  \label{2GKF}
\end{equation}%
here, again we adopted the usual convention $H_{-1}(u)\phi (u):=1-\Phi (u)$;
as in \cite{adlertaylor} we have introduced the flag coefficients%
\begin{equation*}
\left[ 
\begin{array}{c}
2 \\ 
0%
\end{array}%
\right] =\left[ 
\begin{array}{c}
2 \\ 
2%
\end{array}%
\right] =1\text{ , }\left[ 
\begin{array}{c}
2 \\ 
1%
\end{array}%
\right] =\frac{\pi }{2}\text{ ,}
\end{equation*}%
and 
\begin{equation*}
a_{k}(\ell )=\left\{ 
\begin{array}{cc}
O_{p}(\ell ) & \text{for }k=0, \\ 
0 & \text{for }k=1,2%
\end{array}%
\right. \text{ .}
\end{equation*}%
It is important to notice that $\frac{\lambda _{\ell }}{2}=P_{\ell }^{\prime
}(1)$ represents the derivative of the covariance function of random
spherical harmonics at the origin, so that the term%
\begin{equation*}
\frac{\lambda _{\ell }}{2}\int_{\mathbb{S}^{2}}H_{2}(f_{\ell }(x))dx
\end{equation*}%
can be viewed as a (random) measure of the sphere induced by the Riemannian
metric (\ref{Riemetric}); recall indeed that for eigenfunctions $f_{\ell }$
on the sphere $\mathbb{S}^{2}$ the term ${\mathcal{L}}_{2}^{f_{\ell }}(%
\mathbb{S}^{2})$ which appears in (\ref{GKFA}) is exactly given by the area
of the sphere with radius $\left\{ \frac{\lambda _{\ell }}{2}\right\}
^{1/2}, $ i.e., 
\begin{equation*}
{\mathcal{L}}_{2}^{f_{\ell }}(\mathbb{S}^{2})=\frac{\lambda _{\ell }}{2}%
\times 4\pi =\frac{\lambda _{\ell }}{2}\int_{\mathbb{S}^{2}}H_{0}(f_{\ell
}(x))dx\text{ .}
\end{equation*}%
At this stage, it seems very natural to notice that the expected value of
Lipschitz-Killing curvatures can always be written as their projection on
the Wiener chaos of order zero, i.e. in our case 
\begin{equation*}
\mathbb{E}[\mathcal{L}_{k}(A_{u}(f_{\ell };\mathbb{S}^{2}))]=\mathtt{Proj}[%
\mathcal{L}_{k}(A_{u}(f_{\ell };\mathbb{S}^{2}))|0]\text{ ,}
\end{equation*}%
so that we can rewrite the Gaussian Kinematic Formula with an expression
which is remarkably similar to (\ref{2GKF}): 
\begin{equation*}
\mathtt{Proj}[\mathcal{L}_{k}(A_{u}(f_{\ell };\mathbb{S}^{2}))|0]
\end{equation*}%
\begin{equation}
=\left[ 
\begin{array}{c}
2 \\ 
k%
\end{array}%
\right] \left\{ \frac{\lambda _{\ell }}{2}\right\} ^{(2-k)/2}H_{1-k}(u)\phi
(u)\frac{1}{(2\pi )^{(2-k)/2}}\int_{\mathbb{S}^{2}}H_{0}(f_{\ell
}(x))dx+b_{k}(\ell )\text{ ,}  \label{1GKF}
\end{equation}%
where 
\begin{equation*}
b_{k}(\ell )=\left\{ 
\begin{array}{cc}
2(1-\Phi (u))=O(1) & \text{for }k=0, \\ 
0 & \text{for }k=1,2%
\end{array}%
\right. .
\end{equation*}%
The analogy between (\ref{2GKF}) and (\ref{1GKF}) is self-evident; more
explicitly, combining the Gaussian Kinematic Formula with the results from 
\cite{DI}, \cite{MR2015}, \cite{ROSSI2015} and those presented in this paper
we have the following expressions for the projections $\mathtt{Proj}[%
\mathcal{L}_{k}(A_{u}(f_{\ell };\mathbb{S}^{2}))|a],$ $k=0,1,2,$ $a=0,2$:

\emph{a) Excursion Area} ($k=2$)%
\begin{equation*}
\mathtt{Proj}[\mathcal{L}_{2}(A_{u}(f_{\ell };\mathbb{S}^{2}))|0]=\left\{ 
\frac{\lambda _{\ell }}{2}\right\} ^{0}\left[ H_{-1}(u)\phi (u)\right] \int_{%
\mathbb{S}^{2}}H_{0}(f_{\ell }(x))dx\text{ ,}
\end{equation*}%
\begin{equation*}
\mathtt{Proj}[\mathcal{L}_{2}(A_{u}(f_{\ell };\mathbb{S}^{2}))|2]=\frac{1}{2}%
\left\{ \frac{\lambda _{\ell }}{2}\right\} ^{0}\left[ H_{0}(u)H_{1}(u)\phi
(u)\right] \int_{\mathbb{S}^{2}}H_{2}(f_{\ell }(x))dx\text{ ;}
\end{equation*}

\emph{b) (Half) Boundary Length} ($k=1$)%
\begin{equation*}
\mathtt{Proj}[\mathcal{L}_{1}(A_{u}(f_{\ell };\mathbb{S}^{2}))|0]=\left\{ 
\frac{\lambda _{\ell }}{2}\right\} ^{1/2}\sqrt{\frac{\pi }{8}}\left[
H_{0}(u)\phi (u)\right] \int_{\mathbb{S}^{2}}H_{0}(f_{\ell }(x))dx\text{ ,}
\end{equation*}%
\begin{equation*}
\mathtt{Proj}[\mathcal{L}_{1}(A_{u}(f_{\ell };\mathbb{S}^{2}))|2]=\frac{1}{2}%
\left\{ \frac{\lambda _{\ell }}{2}\right\} ^{1/2}\sqrt{\frac{\pi }{8}}\left[
H_{1}^{2}(u)\phi (u)\right] \int_{\mathbb{S}^{2}}H_{2}(f_{\ell }(x))dx\text{
;}
\end{equation*}

\emph{c) Euler-Poncar\'{e} Characteristic} ($k=0$)%
\begin{equation*}
\mathtt{Proj}[\mathcal{L}_{0}(A_{u}(f_{\ell };\mathbb{S}^{2}))|0]=\left\{ 
\frac{\lambda _{\ell }}{2}\right\} \left[ H_{1}(u)\phi (u)\right] \frac{1}{%
2\pi }\int_{\mathbb{S}^{2}}H_{0}(f_{\ell }(x))dx+2\left\{ 1-\Phi (u)\right\} 
\text{ ,}
\end{equation*}%
\begin{equation*}
\mathtt{Proj}[\mathcal{L}_{0}(A_{u}(f_{\ell };\mathbb{S}^{2}))|2]=\frac{1}{2}%
\left\{ \frac{\lambda _{\ell }}{2}\right\} \left[ H_{2}(u)H_{1}(u)\phi (u)%
\right] \frac{1}{2\pi }\int_{\mathbb{S}^{2}}H_{2}(f_{\ell }(x))dx+O_{p}(1)%
\text{ .}
\end{equation*}%
\bigskip 

\subsection{Some comments and conjectures}

We believe that the results we presented in this paper can shed some further
light on a number of geometric features which have been noted in the
literature on random spherical eigenfunctions. In particular, as noted
earlier the asymptotic distribution for each of these Lipschitz-Killing
curvatures is fully degenerate, as it is given by a single (standard
Gaussian) random variable times a deterministic function of the threshold
level $u$. Degeneracy of the limiting distribution provides an easy
explanation for the full asymptotic correlation at different levels $u$
which was earlier noted for the Euler-Poincar\'{e} Characteristic by \cite%
{CMW-EPC}; for the length of level curves this phenomenon was observed in 
\cite{wigmanreview} and addressed in \cite{ROSSI2015}, (see also \cite%
{MPRW2015} for toral eigenfunctions), while for the excursion area
asymptotic degeneracy was established by \cite{DI} and \cite{MR2015}.

On the other hand, as noted already for the case of nodal lines by \cite%
{MPRW2015}, the dominance of the second-order Wiener chaos and its
disappearance for $u=0$ seems to provide a general explanation for the
so-called Berry's cancellation phenomenon (see i.e., \cite{Berry 1977}, \cite%
{Wig}): i.e., the fact that the variance of these geometric functionals is
of lower order in the (`nodal') case $u=0$ than for any other level $u\neq 0$%
. Indeed, the different asymptotic behaviour of these variances is due to
the disappearance of the second-order Wiener chaos term; for the case of
nodal length of arithmetic (toroidal) eigenfunctions, it was shown in \cite%
{MPRW2015} that the fourth-order chaos then dominates, while for the
excursion area the case $u=0$ amounts to the so-called \emph{Defect}, where
all the odd-order chaotic components contribute in the limit (see \cite%
{MW2014}).

We expect these phenomena to hold in greater generality; in particular, we
conjecture that for random eigenfunctions on compact manifolds with
increasing spectral multiplicities the asymptotic behaviour of
Lipschitz-Killing curvatures of excursion sets at any level $u\neq 0$ is
dominated, in the high-energy limit, by the projection on the second-order
Wiener chaos; this leading component appears to vanish in the nodal case $%
u=0 $, hence yielding a phase transition to lower order variance behaviour.
Among the compact manifolds with eigenfunctions which exhibit spectral
degeneracies (i.e., eigenspaces of dimensions larger than one) there are, of
course the sphere $\mathbb{S}^{d}$ and the torus $\mathbb{T}^{d}$ in
arbitrary dimensions $d\geq 2$; a future challenge for research is the
derivation of general expressions akin to (\ref{2GKF}) for the behaviour of
Lipschitz-Killing curvatures in these more general settings.

\subsection{\protect\medskip Plan of the paper}

The plan of this paper is as follows: in Section \ref{Background and
Notation} we review some background material and our notation; Section \ref%
{Wiener2} discusses the projection of the Euler-Poincar\'{e} Characteristic
into second-order chaos, while Section \ref{Variance_CLT} collects the exact
computation of the Variance and the proof of the quantitative Central Limit
Theorem. A number of technical and auxiliary results are collected in the
Appendixes.

\subsection{Acknowledgements}

Research supported by the ERC Grant n. 277742 \emph{Pascal, }(Probabilistic
and Statistical techniques for Cosmological Applications)\emph{. }We are
grateful to Giovanni Peccati, Maurizia Rossi and Igor Wigman for a number of
insightful discussions.\emph{\ }

\section{Background and Notation}

\label{Background and Notation}

\subsection{Morse theory}

As it is customary in this branch of literature, we shall exploit a general
representation for the Euler-Poincar\'{e} Characteristic in terms of
critical points by means of so-called Morse Theory (see \cite{adlertaylor},
Section 9.3). Indeed, assuming that $\mathbb{M}$ is a $C^{2}$ manifold
without boundary in $\mathbb{R}^{N}$ and that $h\in C^{2}(\mathbb{M})$ is a
Morse function on $\mathbb{M}$ (i.e. its Hessian is non-degenerate at the
critical points), it is well-known that the Euler-Poincar\'{e}
Characteristic can be expressed as an alternating sum: 
\begin{equation}
\chi (\mathbb{M})=\sum_{j=0}^{\text{dim}(\mathbb{M})}(-1)^{j}\mu _{j}(%
\mathbb{M},h),  \label{morse_th}
\end{equation}%
where $\mu _{j}(\mathbb{M},h)$ is the number of critical points of $h$ with
Morse index $j$, i.e., the Hessian of $h$ has $j$ negative eigenvalues; for
a proof of (\ref{morse_th}) see \cite{adlertaylor}, Corollary 9.3.3. To
establish our results we will make use of (\ref{morse_th}) in the case of
excursion sets of spherical eigenfunctions; to this aim, we recall some
basic differential geometry on $\mathbb{S}^{2}$, along the same lines as we
did in \cite{CMW-EPC}. More precisely, let us recall that the metric tensor
on the tangent plane $T(\mathbb{S}^{2})$ is given by 
\begin{equation*}
g(\theta ,\varphi )=\left[ 
\begin{matrix}
1 & 0 \\ 
0 & \sin ^{2}\theta%
\end{matrix}%
\right] .
\end{equation*}%
For $x=(\theta ,\varphi )\in \mathbb{S}^{2}\setminus \{N,S\}$ ($N,S$ are the
north and south poles i.e. $\theta =0$ and $\theta =\pi $ respectively), the
vectors 
\begin{equation*}
e_{1}^{x}=\vec{e}_{\theta }=\frac{\partial }{\partial \theta },\hspace{2cm}%
e_{2}^{x}=\vec{e}_{\varphi }=\frac{1}{\sin \theta }\frac{\partial }{\partial
\varphi },
\end{equation*}%
constitute an orthonormal basis for $T_{x}(\mathbb{S}^{2})$; in these system
of coordinates the gradient is given by $\nabla =(\frac{\partial }{\partial
\theta },\frac{1}{\sin \theta }\frac{\partial }{\partial \varphi })$. As
usual, the Hessian of a function $f\in C^{2}(\mathbb{S}^{2})$ is defined as
the bilinear symmetric map from $C^{1}(T(\mathbb{S}^{2}))\times C^{1}(T(%
\mathbb{S}^{2}))$ to $C^{0}(\mathbb{S}^{2})$ given by 
\begin{equation*}
\nabla _{E}^{2}f(X,Y)=XYf-\nabla _{X}Yf,\hspace{1cm}X,Y\in T(\mathbb{S}^{2}),
\end{equation*}%
where $\nabla _{X}$ denotes Levi-Civita connection (see e.g. \cite%
{adlertaylor}, Chapter 7 for more discussion and details). For our
computations to follow we shall need the matrix-valued process $\nabla
_{E}^{2}{f_{\ell }}(x)$ with elements given by 
\begin{equation*}
\{\nabla _{E}^{2}{f_{\ell }}(x)\}_{a,b=\theta ,\varphi }=\{(\nabla ^{2}{%
f_{\ell }}(x))(\vec{e}_{a},\vec{e}_{b})\}_{a,b=\theta ,\varphi },
\end{equation*}%
where $E=\left\{ \vec{e}_{\theta },\vec{e}_{\varphi }\right\} .$ With the
standard system of spherical coordinates, the analytic expression for this
matrix is given by 
\begin{equation*}
\nabla _{E}^{2}f_{\ell }(x)=\left[ 
\begin{matrix}
\frac{\partial ^{2}}{\partial \theta ^{2}}-\Gamma _{\theta \theta }^{\theta }%
\frac{\partial }{\partial \theta }-\Gamma _{\theta \theta }^{\varphi }\frac{%
\partial }{\partial \varphi } & \frac{1}{\sin \theta }[\frac{\partial ^{2}}{%
\partial \theta \partial \varphi }-\Gamma _{\varphi \theta }^{\varphi }\frac{%
\partial }{\partial \varphi }-\Gamma _{\theta \varphi }^{\theta }\frac{%
\partial }{\partial \theta }] \\ 
\frac{1}{\sin \theta }[\frac{\partial ^{2}}{\partial \theta \partial \varphi 
}-\Gamma _{\varphi \theta }^{\varphi }\frac{\partial }{\partial \varphi }%
-\Gamma _{\theta \varphi }^{\theta }\frac{\partial }{\partial \theta }] & 
\frac{1}{\sin ^{2}\theta }[\frac{\partial ^{2}}{\partial \varphi ^{2}}%
-\Gamma _{\varphi \varphi }^{\varphi }\frac{\partial }{\partial \varphi }%
-\Gamma _{\varphi \varphi }^{\theta }\frac{\partial }{\partial \theta }]%
\end{matrix}%
\right]
\end{equation*}%
\begin{equation*}
=\left[ 
\begin{matrix}
\frac{\partial ^{2}}{\partial \theta ^{2}} & \frac{1}{\sin \theta }[\frac{%
\partial ^{2}}{\partial \theta \partial \varphi }-\frac{\cos \theta }{\sin
\theta }\frac{\partial }{\partial \varphi }] \\ 
\frac{1}{\sin \theta }[\frac{\partial ^{2}}{\partial \theta \partial \varphi 
}-\frac{\cos \theta }{\sin \theta }\frac{\partial }{\partial \varphi }] & 
\frac{1}{\sin ^{2}\theta }[\frac{\partial ^{2}}{\partial \varphi ^{2}}+\sin
\theta \cos \theta \frac{\partial }{\partial \theta }]%
\end{matrix}%
\right] ,
\end{equation*}%
where $\Gamma _{ab}^{c}$ are the usual Christoffel symbols, see e.g. \cite%
{chavel} Section I.1, from which we can compute the Levi-Civita connection: 
\begin{equation*}
\nabla _{\vec{e}_{a}}\vec{e}_{b}=\Gamma _{ab}^{\theta }\vec{e}_{\theta
}+\Gamma _{ab}^{\varphi }\vec{e}_{\varphi },\hspace{0.7cm}a,b=\theta
,\varphi .
\end{equation*}%
More explicitly, Christoffel symbols for $\mathbb{S}^{2}$ are given by 
\begin{equation*}
\Gamma _{\theta \varphi }^{\theta }=\Gamma _{\theta \theta }^{\theta
}=\Gamma _{\varphi \varphi }^{\varphi }=\Gamma _{\theta \theta }^{\varphi
}=0,\hspace{0.5cm}\Gamma _{\varphi \varphi }^{\theta }=-\sin \theta \cos
\theta ,\hspace{0.5cm}\Gamma _{\varphi \theta }^{\varphi }=\cot \theta .
\end{equation*}

For every $x\in \mathbb{S}^{2},$ let $\nabla f_{\ell }(x)$ and $\nabla
^{2}f_{\ell }(x)$ be the vector-valued processes with elements 
\begin{equation*}
\nabla f_{\ell }(x)=(e_{1}^{x}f_{\ell }(x),e_{2}^{x}f_{\ell }(x)),\hspace{%
0.3cm}\nabla ^{2}f_{\ell }(x)=(e_{1}^{x}e_{1}^{x}f_{\ell
}(x),e_{1}^{x}e_{2}^{x}f_{\ell }(x),e_{2}^{x}e_{2}^{x}f_{\ell }(x)).
\end{equation*}%
Since the $f_{\ell }$ are eigenfunctions of the spherical Laplacian, the
value of $f_{\ell }$ at every fixed point $x\in \mathbb{S}^{2}$ is a linear
combination of its first and second order derivatives at $x$. If the point $%
x\in \mathbb{S}^{2}$ is also a critical point for $f_{\ell }$ it follows
that the value of the spherical harmonic at $x$ is a linear combination of
its second order derivatives, i.e., 
\begin{equation}
e_{1}^{x}e_{1}^{x}f_{\ell }(x)+e_{2}^{x}e_{2}^{x}f_{\ell }(x)=-\lambda
_{\ell }f_{\ell }(x).  \label{linear_dp}
\end{equation}%
Let us take $\mathbb{M}$ and $h$ in formula (\ref{morse_th}) to be $%
A_{u}(f_{\ell };\mathbb{S}^{2})$ and $\left. f_{\ell }\right\vert _{{%
A_{u}(f_{\ell };\mathbb{S}^{2})}}$ respectively; by the Morse
representation, we obtain 
\begin{equation}
\chi (A_{u}(f_{\ell };\mathbb{S}^{2}))=\sum_{j=0}^{2}(-1)^{j}\mu _{j},
\label{morse}
\end{equation}%
where 
\begin{equation*}
\mu _{j}=\#\{x\in \mathbb{S}^{2}:f_{\ell }(x)\geq u,\nabla f_{\ell }(x)=0,%
\text{Ind}(-\nabla ^{2}f_{\ell }(x))=j\}
\end{equation*}%
\begin{equation*}
=\#\{x\in \mathbb{S}^{2}:e_{1}^{x}e_{1}^{x}f_{\ell
}(x)+e_{2}^{x}e_{2}^{x}f_{\ell }(x)\leq -\lambda _{\ell }\,u,\nabla f_{\ell
}(x)=0,\text{Ind}(-\nabla ^{2}f_{\ell }(x))=j\},
\end{equation*}%
$\text{Ind}(M)$ denoting the number of negative eigenvalues of a square
matrix $M$. More specifically, $\mu _{0}$ is the number of maxima, $\mu _{1}$
the number of saddles, and $\mu _{2}$ the number of minima in the excursion
region $A_{u}(f_{\ell };\mathbb{S}^{2})$. In the next subsection, we show
how to justify this representation into a $L^{2}$ space, by means of an
approximating sequence of delta functions.\medskip

\subsection{The delta function approximation}

\noindent Let us now denote by $\Sigma _{\ell }(x,y)$ the covariance matrix
for the $10$-dimensional Gaussian random vector 
\begin{equation*}
(\nabla f_{\ell }(x),\nabla f_{\ell }(y),\nabla ^{2}f_{\ell }(x),\nabla
^{2}f_{\ell }(y))
\end{equation*}%
which combines the gradient and the elements of the Hessian evaluated at $x$%
, $y$; we shall write 
\begin{equation*}
\Sigma _{\ell }(x,y)=\left( 
\begin{array}{cc}
A_{\ell }(x,y) & B_{\ell }(x,y) \\ 
B_{\ell }^{t}(x,y) & C_{\ell }(x,y)%
\end{array}%
\right) ,
\end{equation*}%
where the $A_{\ell }$ and $C_{\ell }$ components collect the variances of
the gradient and Hessian terms respectively, while the matrix $B_{\ell }$
collects the covariances between first and second order derivatives. The
explicit computation of $\Sigma _{\ell }(x,y)$ requires iterative
derivations of Legendre polynomials and are given in \cite{CMW}, Appendix 1.
For the $L^{2}$ expansion of the Euler-Poincar\'{e} Characteristic to hold,
we need to assume the following, standard non-degeneracy condition :

\begin{condition}
\label{Thecondition} For every $(x,y)\in \mathbb{S}^{2}$, the Gaussian
vector $(\nabla f_{\ell }(x),\nabla f_{\ell }(y))$ has a non-degenerate
density function, i.e., the covariance matrix $A_{\ell }(x,y)$ is invertible.
\end{condition}

\medskip We can now build an approximating sequence of delta functions, and
establish their convergence both in the a.s. and in the $L^{2}$ sense. More
precisely, let $\delta _{\varepsilon }:\mathbb{R}^{2}\rightarrow \mathbb{R}$
be such that 
\begin{equation*}
\delta _{\varepsilon }(z)=(2\varepsilon )^{-2}\mathbb{I}_{[-\varepsilon
,\varepsilon ]^{2}}(z),
\end{equation*}%
and define the approximating sequence 
\begin{equation*}
\chi _{\varepsilon }(A_{u}(f_{\ell };\mathbb{S}^{2}))=\sum_{j=0}^{2}\mu
_{j}(\varepsilon ),
\end{equation*}%
where 
\begin{equation*}
\mu _{j}(\varepsilon )=\int_{\mathbb{S}^{2}}|\mathrm{det}(\nabla ^{2}f_{\ell
}(x))|\mathbb{I}_{\{\widetilde{f}_{\ell }(x)\geq u\}}\mathbb{I}_{\{\text{Ind}%
(-\nabla ^{2}f_{\ell }(x))=j\}}\delta _{\varepsilon }(\nabla f_{\ell }(x))dx,
\end{equation*}%
and we wrote for brevity%
\begin{equation*}
\widetilde{f}_{\ell }(x):=-\frac{e_{1}^{x}e_{1}^{x}f_{\ell
}(x)+e_{2}^{x}e_{2}^{x}f_{\ell }(x)}{\lambda _{\ell }};
\end{equation*}%
note that $\widetilde{f}_{\ell }(x)=f_{\ell }(x)$ when $x$ is a critical
point, i.e., as $\varepsilon \rightarrow 0$. Now recall the standard
identity (see i.e., \cite{adlerstflour}, Lemma 4.2.2) 
\begin{equation}
\sum_{j=0}^{2}(-1)^{j}\;|\det (\nabla ^{2}f_{\ell }(x))|\mathbb{I}_{\{\text{%
Ind}(-\nabla ^{2}f_{\ell }(x))=j\}}=\det (-\nabla ^{2}f_{\ell }(x))
\label{simp1}
\end{equation}%
\ so that we can rewrite $\chi _{\varepsilon }(A_{u}(f_{\ell };\mathbb{S}%
^{2}))$ as%
\begin{equation*}
\chi _{\varepsilon }(A_{u}(f_{\ell };\mathbb{S}^{2}))=\int_{\mathbb{S}^{2}}%
\mathrm{det}(\nabla ^{2}f_{\ell }(x))\mathbb{I}_{\{\widetilde{f}_{\ell
}(x)\geq u\}}\delta _{\varepsilon }(\nabla f_{\ell }(x))dx.
\end{equation*}%
As in \cite{estradeleon}, we are now able to prove the almost sure and $%
L^{2}(\Omega )$ convergence of $\chi _{\varepsilon }(A_{u}(f_{\ell };\mathbb{%
S}^{2}))$ to $\chi (A_{u}(f_{\ell };\mathbb{S}^{2})),$ as $\varepsilon
\rightarrow 0:$

\begin{lemma}
\label{XXe} For every $\ell $ such that Condition \ref{Thecondition} holds,
we have 
\begin{equation}
\chi (A_{u}(f_{\ell };\mathbb{S}^{2}))=\lim_{\varepsilon \rightarrow 0}\chi
_{\varepsilon }(A_{u}(f_{\ell };\mathbb{S}^{2})),  \label{Xe1}
\end{equation}%
where the convergence holds both $\omega -$a.s. and in $L^{2}(\Omega ).$
\end{lemma}

\begin{proof}
To prove almost sure convergence, we first apply \cite{adlertaylor}, Theorem
11.2.3, where we take $f=\nabla f_{\ell }:\mathbb{S}^{2}\rightarrow \mathbb{R%
}^{2}$, $g=(f_{\ell },-\nabla ^{2}f_{\ell }):\mathbb{S}^{2}\rightarrow 
\mathbb{R}^{4}$, $u=0$ and $B=B_{j}=[u,\infty )\times \{\text{Ind}=j\}$, so
that, for $j=0,1,2$, we have 
\begin{equation}
\mu _{j}=\lim_{\varepsilon \rightarrow 0}\mu _{j}(\varepsilon ),\hspace{1cm}%
\omega -{\text{a}}\text{{.s.}}  \label{mj1}
\end{equation}%
We note that the conditions in \cite{adlertaylor}, Theorem 11.2.3, are all
fulfilled since random spherical harmonics are Morse functions with
probability one, under Condition \ref{Thecondition}; then the almost sure
convergence (\ref{Xe1}) immediately follows from (\ref{morse}), (\ref{mj1})
and (\ref{simp1}). We prove now that (\ref{Xe1}) also holds in $L^{2}(\Omega
)$; it is a classical result that $L^{2}$-convergence follows from
convergence a.s. and convergence of the $L^{2}$ norm, whence the proof will
be completed if we show that 
\begin{equation}
\lim_{\varepsilon \rightarrow 0}\mathbb{E}[\chi _{\varepsilon
}(A_{u}(f_{\ell };\mathbb{S}^{2}))]^{2}=\mathbb{E}[\chi (A_{u}(f_{\ell };%
\mathbb{S}^{2}))]^{2}.  \label{L2}
\end{equation}%
Indeed, note that 
\begin{equation*}
\mathbb{E}[\chi _{\varepsilon }(A_{u}(f_{\ell };\mathbb{S}%
^{2}))]^{2}=\sum_{j,k=0}^{2}(-1)^{j+k}\mathbb{E}[\mu _{j}(\varepsilon )\mu
_{k}(\varepsilon )].
\end{equation*}%
Under Condition \ref{Thecondition} we can apply Kac-Rice formula to compute $%
\mathbb{E}[\mu _{j}(\varepsilon )\mu _{k}(\varepsilon )]$ (see \cite%
{azaiswschebor}, Theorem 6.3 or \cite{adlertaylor}, Theorem 11.2.1) and,
proceeding as in the proof of \cite{CMW-EPC}, Proposition 1, we obtain 
\begin{equation*}
\sum_{j,k=0}^{2}(-1)^{j+k}\mathbb{E}[\mu _{j}(\varepsilon )\mu
_{k}(\varepsilon )]=\int_{\mathbb{S}^{2}}\int_{\mathbb{S}^{2}}\int_{u}^{%
\infty }\int_{u}^{\infty }J_{2,\ell ,\varepsilon
}(x,y,t_{1},t_{2})dt_{1}dt_{2}dxdy,
\end{equation*}%
where%
\begin{equation*}
J_{2,\ell ,\varepsilon }(x,y,t_{1},t_{2})
\end{equation*}%
\begin{equation*}
=\frac{1}{(2\varepsilon )^{4}}\iint_{[-\varepsilon ,\varepsilon ]^{2}\times
\lbrack -\varepsilon ,\varepsilon ]^{2}}\left\{ \varphi _{(\tilde{f}_{\ell
}(x),\tilde{f}_{\ell }(y),\nabla f_{\ell }(x),\nabla f_{\ell
}(y))}(t_{1},t_{2},\eta _{1},\eta _{2})\right.
\end{equation*}%
\begin{equation*}
\left. \times \mathbb{E}[\text{det}(-\nabla ^{2}f_{\ell }(x))\text{det}%
(-\nabla ^{2}f_{\ell }(y))\big|\nabla f_{\ell }(x)=\eta _{1},\nabla f_{\ell
}(y)=\eta _{2},\tilde{f}_{\ell }(x)=t_{1},\tilde{f}_{\ell
}(y)=t_{2}]\right\} d\eta _{1}d\eta _{2}\text{ },
\end{equation*}%
and $\varphi _{(\tilde{f}_{\ell }(x),\tilde{f}_{\ell }(y),\nabla f_{\ell
}(x),\nabla f_{\ell }(y))}$ is the density of the $6$-dimensional vector 
\begin{equation*}
(\tilde{f}_{\ell }(x),\tilde{f}_{\ell }(y),\nabla f_{\ell }(x),\nabla
f_{\ell }(y)).
\end{equation*}%
We note also that, under Condition \ref{Thecondition}, the covariance matrix 
$A_{\ell }(x,y)$ and the conditional covariance matrix of the Gaussian
vector 
\begin{equation*}
(\nabla ^{2}f_{\ell }(x),\nabla ^{2}f_{\ell }(y)\big|\nabla f_{\ell
}(x),\nabla f_{\ell }(y),\tilde{f}_{\ell }(x),\tilde{f}_{\ell }(y))
\end{equation*}%
%
%
%
%
%
%
%$\Omega _{\ell}(\phi )=C_{\ell }(\phi )-B_{\ell }^{t}(\phi )A_{\ell }^{-1}(\phi )B_{\ell
%}(\phi )$
are invertible for $x,y\in \mathbb{S}^{2}$; hence the conditional Gaussian
density function is continuous and thus, as $\varepsilon \rightarrow 0$, the
integral $J_{2,\ell ,\varepsilon }(x,y,t_{1},t_{2})$ converges to%
\begin{equation*}
J_{2,\ell }(x,y,t_{1},t_{2})=\varphi _{(f_{\ell }(x),f_{\ell }(y),\nabla
f_{\ell }(x),\nabla f_{\ell }(y))}(t_{1},t_{2},0,0)
\end{equation*}%
\begin{equation*}
\times \mathbb{E}[\text{det}(-\nabla ^{2}f_{\ell }(x))\text{det}(-\nabla
^{2}f_{\ell }(y))\big|f_{\ell }(x)=t_{1},f_{\ell }(y)=t_{2},\nabla f_{\ell
}(x)=0,\nabla f_{\ell }(y)=0].
\end{equation*}%
The statement follows by observing that under Condition \ref{Thecondition},
and in view of \cite{CMW-EPC}, Proposition 1, we also have 
\begin{equation*}
\mathbb{E}[\chi (A_{u}(f_{\ell };\mathbb{S}^{2}))]^{2}=\int_{\mathbb{S}%
^{2}}\int_{\mathbb{S}^{2}}\int_{u}^{\infty }\int_{u}^{\infty }J_{2,\ell
}(x,y,t_{1},t_{2})dt_{1}dt_{2}dxdy.
\end{equation*}
\end{proof}

\subsection{Wiener Chaos}

%\emph{@@Per Valentina - per favore, puoi cominciare a scrivere qualcosa per
%questa sottosezione? abbiamo bisogno solo di ricordare qualche definizione,
%puoi prenderle magari da qualche articolo precedente, incluso l'ultimo sul
%toro con maurizia, giovanni ed igor@@@}
%

In this section we recall very briefly some basic facts on Wiener-It\^{o}
chaotic expansion for non-linear functionals of Gaussian fields. We follow
closely the summary which was given in \cite{MPRW2015}, while we refer to 
\cite{noupebook} for an exhaustive discussion.

Recall first that each random eigenfunction $f_{\ell }$ in (\ref{fell}) is a
by-product of the family of complex-valued, independent, Gaussian random
variables $\{a_{\ell m}\}$, $m=-\ell ,\dots ,\ell $, defined on some
probability space $(\Omega ,{\mathcal{F}},\mathbb{P})$ and satisfying the
following properties: i) for $m\neq 0$ every $a_{\ell m}$ has the form 
\begin{equation*}
\mathrm{Re}(a_{\ell m})+i\,\mathrm{Im}(a_{\ell m}),
\end{equation*}%
where $\mathrm{Re}(a_{\ell m})$ and $\mathrm{Im}(a_{\ell m})$ are two
zero-mean, independent Gaussian variables with variance ${1}/{2}$; ii) $%
a_{\ell 0}$ follows a standard Gaussian distribution; iii) $a_{\ell ,m}$ and 
$a_{\ell ,m^{\prime }}$ are stochastically independent whenever $m^{\prime
}\neq -m$; iv) $(-1)^{m}a_{\ell ,-m}=\bar{a}_{\ell m}$. We define the space $%
\mathbf{A}$ to be the closure in $L^{2}(\mathbb{P})$ of all real finite
linear combinations of random variables of the forms 
\begin{equation*}
z(-1)^{m}a_{\ell ,-m}+\bar{z}\,a_{\ell m}\hspace{0.5cm}\text{and}\hspace{%
0.5cm}a_{\ell 0},
\end{equation*}%
$z\in \mathbb{C}$; the space $A$ is a real, centred, Gaussian Hilbert
subspace of $L^{2}(\mathbb{P})$. For each $q\geq 0$ the $q$-th \textit{%
Wiener chaos} ${\mathcal{H}_{q}}$ associated with $\mathbf{A}$ is the closed
linear subspace of $L^{2}(\mathbb{P})$ generated by all real, finite, linear
combinations of random variables of the form 
\begin{equation*}
H_{q_{1}}(x_{1})\cdot H_{q_{2}}(x_{2})\cdots H_{q_{k}}(x_{k})
\end{equation*}%
for $k\geq 1$, where the integers $q_{1},q_{2},\dots ,q_{k}\geq 0$ satisfy $%
q_{1}+q_{2}+\cdots +q_{k}=q$, and $(x_{1},x_{2},\dots ,x_{k})$ is a
standard, real, Gaussian vector extracted form $\mathbf{A}$; note that in
particular ${\mathcal{H}_{0}}=\mathbb{R}$. As well-known Wiener chaoses $%
\left\{ {\mathcal{H}_{q},}\text{ }q=0,1,2,...\right\} $ are orthogonal,
i.e., ${\mathcal{H}_{q}}\bot {\mathcal{H}_{p}}$ for $p\neq q$; moreover, the
following \textit{Wiener-It{\^{o}}} decomposition of $L^{2}(\mathbb{P})$
holds: every random variable $F\in L^{2}(\mathbb{P})$ admits a unique
expansion of the type 
\begin{equation}
F=\mathbb{E}[F]+\sum_{q=1}^{\infty }\mathtt{Proj}[F|q]  \label{chaos}
\end{equation}
where the projections $\mathtt{Proj}[F|q]\in \mathcal{H}_{q}$ for every $%
q=1,2,...$and the series converges in $L^{2}(\mathbb{P})$. Again we refer 
\cite{noupebook}, Theorem 2.2.4, for an extremely rich discussion and a vast
gallery of examples and applications.

\subsection{Overview of the Proof}

The main technical tools for our argument are collected in Proposition \ref%
{hk} and Proposition \ref{AB}; the proof of each of these results takes a
separate Section in the Appendix. In particular, in Proposition \ref{hk} we
derive explicit analytic expression for the projection coefficients on the
components of second order Wiener chaos; in Proposition \ref{AB}, we manage
to write down the integrals over the sphere of these components in terms of
weighted sums of the random spherical harmonic coefficients $\left\{ a_{\ell
m}\right\} $: the latter results requires a very careful analytic
investigation on derivatives of Associated Legendre Function, which is given
in the third Section of the Appendix. Combining together Proposition \ref{hk}
and Proposition \ref{AB}, one obtains an explicit formula for the
second-order Wiener chaos, which can be further simplified by some algebraic
manipulations to achieve the statement of Theorem \ref{proj2}. Because the
spherical harmonic coefficients are independent and identically distributed
(excluding the term at $m=0$), the conclusions of Theorem \ref{clt} are then
rather straightforward to obtain.

\section{The Projection into the second Wiener Chaos}

\label{Wiener2}

In this section we prove Theorem \ref{proj2}, i.e., we derive an analytic
expression for the projection of the Euler-Poincar\'{e} Characteristic on
the second-order Wiener chaos. Our strategy for this proof can be summarized
as follows: from standard results in Morse theory detailed in the previous
Section, we can express the Euler-Poincar\'{e} Characteristic as a function
of a six-dimensional vector, involving the eigenfunctions $f_{\ell },$ the
two-dimensional gradient vector, and the three-dimensional vector including
the independent components of the Hessian. Actually, as in \cite{CMW} these
components may immediately be reduced to five, as the eigenfunctions can be
written as linear combinations of first and second order derivatives. It is
then convenient to implement a linear transform on this vector, to make its
components independent when evaluated on the same point $x\in \mathbb{S}^{2}$%
; this idea is analogous to the approach which was pursue by \cite%
{estradeleon} in their recent work on the Euler-Poincar\'{e} Characteristic
for Gaussian field on an Euclidean (growing) domain. We are then able to
write down explicitly the projection coefficients on the second-order Wiener
chaos; the result then follows from a very careful cancellation of the
different projection components.

\subsection{Cholesky decomposition}

In view of (\ref{linear_dp}) it follows that we can rewrite $\chi
_{\varepsilon }(A_{u}(f_{\ell };\mathbb{S}^{2}))$ as 
\begin{align*}
\chi _{\varepsilon }(A_{u}(f_{\ell };\mathbb{S}^{2}))& =\int_{\mathbb{S}%
^{2}}[e_{1}^{x}e_{1}^{x}f_{\ell }(x)\;e_{2}^{x}e_{2}^{x}f_{\ell
}(x)-(e_{1}^{x}e_{2}^{x}f_{\ell }(x))^{2}] \\
& \;\;\times \mathbb{I}_{\{e_{1}^{x}e_{1}^{x}f_{\ell
}(x)+e_{2}^{x}e_{2}^{x}f_{\ell }(x)\leq -\lambda _{\ell }u\}}\delta
_{\varepsilon }(e_{1}^{x}f_{\ell }(x),e_{2}^{x}f_{\ell }(x))dx.
\end{align*}%
It should be noted that the integrand 
\begin{equation*}
\lbrack e_{1}^{x}e_{1}^{x}f_{\ell }(x)\;e_{2}^{x}e_{2}^{x}f_{\ell
}(x)-(e_{1}^{x}e_{2}^{x}f_{\ell }(x))^{2}]\mathbb{I}_{%
\{e_{1}^{x}e_{1}^{x}f_{\ell }(x)+e_{2}^{x}e_{2}^{x}f_{\ell }(x)\leq -\lambda
_{\ell }u\}}\delta _{\varepsilon }(e_{1}^{x}f_{\ell }(x),e_{2}^{x}f_{\ell
}(x))
\end{equation*}%
is isotropic, so focussing on the great circle $\theta _{x}=\frac{\pi }{2}$
is simply a convenient simplification. Let us now write $\sigma _{\ell }(x)$
for the $5\times 5$ covariance matrix of the Gaussian random vector 
\begin{equation*}
(e_{1}^{x}f_{\ell }(x),e_{2}^{x}f_{\ell }(x),e_{1}^{x}e_{1}^{x}f_{\ell
}(x),e_{1}^{x}e_{2}^{x}f_{\ell }(x),e_{2}^{x}e_{2}^{x}f_{\ell }(x)),
\end{equation*}%
i.e. the $5\times 1$ vector that includes the gradient and the Hessian
components of interest. We evaluate the covariance matrix $\sigma _{\ell
}(x) $ on the great circle such that $\theta _{x}=\frac{\pi }{2}$, and we
write it in the partitioned form 
\begin{equation*}
\sigma _{\ell }(x)_{5\times 5}=\left( 
\begin{array}{cc}
a_{\ell }(x) & b_{\ell }(x) \\ 
b_{\ell }^{t}(x) & c_{\ell }(x)%
\end{array}%
\right) ,
\end{equation*}%
where the superscript $t$ denotes transposition, and (see \cite{CMW},
Section 2.2) 
\begin{equation*}
a_{\ell }(\pi /2,\varphi )=\left( 
\begin{array}{cc}
\frac{\lambda _{\ell }}{2} & 0 \\ 
0 & \frac{\lambda _{\ell }}{2}%
\end{array}%
\right) ,\hspace{1cm}b_{\ell }(\pi /2,\varphi )=\left( 
\begin{array}{ccc}
0 & 0 & 0 \\ 
0 & 0 & 0%
\end{array}%
\right) ,
\end{equation*}%
\begin{equation*}
c_{\ell }(\pi /2,\varphi )=\frac{\lambda _{\ell }^{2}}{8}\left( 
\begin{array}{ccc}
3-\frac{2}{\lambda _{\ell }} & 0 & 1+\frac{2}{\lambda _{\ell }} \\ 
0 & 1-\frac{2}{\lambda _{\ell }} & 0 \\ 
1+\frac{2}{\lambda _{\ell }} & 0 & 3-\frac{2}{\lambda _{\ell }}%
\end{array}%
\right) .
\end{equation*}%
\medskip

\noindent Let us first recall that the \emph{Cholesky decomposition} of a
Hermitian positive-definite matrix $A$ takes the form $A=\Lambda \Lambda
^{t},$ where $\Lambda $ is a lower triangular matrix with real and positive
diagonal entries, and $\Lambda ^{t}$ denotes the conjugate transpose of $%
\Lambda $. It is well-known that every Hermitian positive-definite matrix
(and thus also every real-valued symmetric positive-definite matrix) admits
a unique Cholesky decomposition.

By an explicit computation, it is then possible to show that the Cholesky
decomposition of $\sigma _{\ell }$ takes the form $\sigma _{\ell }=\Lambda
_{\ell }\Lambda _{\ell }^{t},$ where 
\begin{equation*}
\Lambda _{\ell }=\left( 
\begin{array}{ccccc}
\frac{\sqrt{\lambda }_{\ell }}{\sqrt{2}} & 0 & 0 & 0 & 0 \\ 
0 & \frac{\sqrt{\lambda }_{\ell }}{\sqrt{2}} & 0 & 0 & 0 \\ 
0 & 0 & \frac{\sqrt{\lambda _{\ell }}\sqrt{3\lambda _{\ell }-2}}{2\sqrt{2}}
& 0 & 0 \\ 
0 & 0 & 0 & \frac{\sqrt{\lambda _{\ell }}\sqrt{\lambda _{\ell }-2}}{2\sqrt{2}%
} & 0 \\ 
0 & 0 & \frac{\sqrt{\lambda _{\ell }}(\lambda _{\ell }+2)}{2\sqrt{2}\sqrt{%
3\lambda _{\ell }-2}} & 0 & \frac{\lambda _{\ell }\sqrt{{\lambda _{\ell }-2}}%
}{\sqrt{3\lambda _{\ell }-2}}%
\end{array}%
\right)
\end{equation*}%
\begin{equation*}
=:\left( 
\begin{array}{ccccc}
\lambda _{1} & 0 & 0 & 0 & 0 \\ 
0 & \lambda _{1} & 0 & 0 & 0 \\ 
0 & 0 & \lambda _{3} & 0 & 0 \\ 
0 & 0 & 0 & \lambda _{4} & 0 \\ 
0 & 0 & \lambda _{2} & 0 & \lambda _{5}%
\end{array}%
\right) ;
\end{equation*}%
in the last expression, for notational simplicity we have omitted the
dependence of the $\lambda _{i}$s on $\ell $. The matrix is block diagonal,
because under isotropy the gradient components are independent from the
Hessian when evaluated at the same point (see i.e., \cite{adlertaylor},
Section 5.5). We can hence define a $5$-dimensional standard Gaussian vector 
$Y(x)=(Y_{1}(x),Y_{2}(x),Y_{3}(x),Y_{4}(x),Y_{5}(x))$ with independent
components such that 
\begin{equation*}
(e_{1}^{x}f_{\ell }(x),e_{2}^{x}f_{\ell }(x),e_{1}^{x}e_{1}^{x}f_{\ell
}(x),e_{1}^{x}e_{2}^{x}f_{\ell }(x),e_{2}^{x}e_{2}^{x}f_{\ell }(x))
\end{equation*}%
\begin{equation*}
=\Lambda _{\ell }Y(x)=\left( \lambda _{1}Y_{1}(x),\lambda
_{1}Y_{2}(x),\lambda _{3}Y_{3}(x),\lambda _{4}Y_{4}(x),\lambda
_{5}Y_{5}(x)+\lambda _{2}Y_{3}(x)\right) .
\end{equation*}%
The expression that we need to expand can then be written as 
\begin{equation*}
\lbrack e_{1}^{x}e_{1}^{x}f_{\ell }(x)\;e_{2}^{x}e_{2}^{x}f_{\ell
}(x)-(e_{1}^{x}e_{2}^{x}f_{\ell }(x))^{2}]\;\mathbb{I}_{%
\{e_{1}^{x}e_{1}^{x}f_{\ell }(x)+e_{2}^{x}e_{2}^{x}f_{\ell }(x)\leq -\lambda
_{\ell }u\}}\;\delta _{\varepsilon }(e_{1}^{x}f_{\ell
}(x),\,e_{2}^{x}f_{\ell }(x))
\end{equation*}%
\begin{equation*}
=[\lambda _{3}Y_{3}(x)\{\lambda _{5}Y_{5}(x)+\lambda
_{2}Y_{3}(x)\}-\{\lambda _{4}Y_{4}(x)\}^{2}]\;\mathbb{I}_{\left\{ \frac{%
\lambda _{3}}{\lambda }Y_{3}(x)+\frac{\lambda _{5}}{\lambda }Y_{5}(x)+\frac{%
\lambda _{2}}{\lambda }Y_{3}(x)\leq -u\right\} }\;\delta _{\varepsilon
}(\lambda _{1}Y_{1}(x),\,\lambda _{1}Y_{2}(x)).
\end{equation*}

\subsection{Second order chaotic component}

We need now to start computing the projection coefficients on second-order
Wiener chaoses. Our notation is as follows; we write \noindent $h_{ij}$, $%
i,j=1,\dots 5$, $i\neq j$, for the projections on terms of the form $%
H_{1}(Y_{i})H_{1}(Y_{j})=Y_{i}Y_{j},$ i.e., we define 
\begin{equation*}
h_{ij}(u;\ell )=\lim_{\varepsilon \to 0} \mathbb{E}\left[[ \lambda _{3}Y_{3}
\{\lambda _{5}Y_{5}+\lambda _{2}Y_{3}\}-(\lambda _{4}Y_{4})^{2}]\;1\hspace{%
-0.27em}\mbox{\rm l}_{\left\{ \frac{\lambda _{2}+\lambda _{3}}{\lambda }%
Y_{3}+\frac{\lambda _{5}}{\lambda }Y_{5}\leq -u\right\}
}\;\delta_{\varepsilon} (\lambda _{1}Y_{1}, \lambda _{1}Y_{2})Y_{i}Y_{j}%
\right];
\end{equation*}%
on the other hand, we write $k_{i}$, $i=1,\dots 5$, for the projection on
terms of the form $H_{2}(Y_{i})$, i.e., we define 
\begin{equation*}
k_{i}(u;\ell )=\lim_{\varepsilon \to 0} \mathbb{E}\left[ [\lambda _{3}Y_{3}
\{\lambda _{5}Y_{5}+\lambda _{2}Y_{3}\}-(\lambda _{4}Y_{4})^{2}]\;1\hspace{%
-0.27em}\mbox{\rm l}_{\left\{ \frac{\lambda _{2}+\lambda _{3}}{\lambda }%
Y_{3}+\frac{\lambda _{5}}{\lambda }Y_{5}\leq -u\right\}
}\;\delta_{\varepsilon} (\lambda _{1}Y_{1}, \lambda _{1}Y_{2})H_{2}(Y_{i})%
\right].
\end{equation*}

\noindent The second order chaotic component of the Euler-Poincar\'{e}
Characteristic is then given by 
\begin{equation*}
\mathtt{Proj}[\chi (A_{u}(f_{\ell };\mathbb{S}^{2}))|2]=\sum_{i=1}^{5}%
\sum_{j=1}^{i}h_{ij}(u;\ell )\int_{\mathbb{S}^{2}}Y_{i}(x)Y_{j}(x)dx+\frac{1%
}{2}\sum_{i=1}^{5}k_{i}(u;\ell )\int_{\mathbb{S}^{2}}H_{2}(Y_{i}(x))dx.
\end{equation*}%
The following Proposition provides analytic expressions for the coefficients 
$h_{ij}$ and $k_{i}$:

\begin{proposition}
\label{hk} a) All coefficients $h_{ij}(u;\ell )$ are identically zero,
unless ($i,j)=(3,5),$ i.e. 
\begin{equation*}
h_{ij}(u;\ell )=\sqrt{\lambda _{\ell }}\sqrt{\lambda _{\ell }-2}\frac{\Phi
(-u)(3\lambda _{\ell }-2)+u\phi (u)[2+\lambda _{\ell }(u^{2}+1)]}{2\sqrt{2}%
\pi (3\lambda _{\ell }-2)}\delta _{i}^{3}\delta _{j}^{5};
\end{equation*}
b) For the $k_{i}$ coefficients we have%
\begin{equation*}
k_{1}(u;\ell )=k_{2}(u;\ell )=-\frac{2\Phi (-u)+\lambda _{\ell }u\phi (u)}{%
4\pi },
\end{equation*}%
\begin{equation*}
k_{3}(u;\ell )=\Phi (-u)\frac{\lambda _{\ell }+2}{4\pi }+\lambda_{\ell} 
\frac{2+\lambda _{\ell }(u^{2}+1)}{2\pi (3\lambda _{\ell }-2)}u\phi (u),
\end{equation*}
\begin{equation*}
k_{4}(u;\ell )=-\Phi (-u)\frac{\lambda _{\ell }-2}{4\pi },\text{ }%
k_{5}(u;\ell )=(\lambda _{\ell }-2)\frac{\lambda _{\ell }(u^{2}+1)+2}{4\pi
(3\lambda _{\ell }-2)}u\phi (u).
\end{equation*}
\end{proposition}

The proof of Proposition \ref{hk} is postponed to the Appendix \ref{proofhk}%
. From Proposition \ref{hk} it is then immediate to obtain the following
expression: 
\begin{equation*}
\mathtt{Proj}[\chi (A_{u}(f_{\ell };\mathbb{S}^{2}))|2]=h_{35}(u;\ell
)A_{35}(\ell )+\frac{1}{2}\sum_{i=1}^{5}k_{i}(u;\ell )B_{i}(\ell ).
\end{equation*}%
where 
\begin{equation*}
A_{ij}(\ell )=\int_{\mathbb{S}^{2}}Y_{i}(x)Y_{j}(x)dx,\hspace{1cm}B_{i}(\ell
)=\int_{\mathbb{S}^{2}}H_{2}(Y_{i}(x))dx.
\end{equation*}%
Our next step is then to investigate the behaviour of these integrals of
stochastic processes; this task is accomplished in the following Lemma.

\begin{proposition}
\label{AB} We have that 
\begin{equation*}
A_{35}(\ell )=4\pi \frac{\sqrt{2}}{3}\sum_{m=-\ell }^{\ell }\{|a_{\ell
m}|^{2}-1\}\left[ -\frac{1}{\ell }+\frac{3\,m}{\ell ^{2}}-\frac{2\,m^{3}}{%
\ell ^{4}}\right] +r_{0}(\ell ),
\end{equation*}%
and moreover%
\begin{equation*}
B_{1}(\ell )=4\pi \,\sum_{m=-\ell }^{\ell }\{|a_{\ell m}|^{2}-1\}\left[ 
\frac{1}{\ell }-\frac{m}{\ell ^{2}}\right] +r_{1}(\ell ),
\end{equation*}%
\begin{equation*}
B_{2}(\ell )=4\pi \,\sum_{m=-\ell }^{\ell }\{|a_{\ell m}|^{2}-1\}\frac{m}{%
\ell ^{2}}+r_{2}(\ell ),
\end{equation*}%
\begin{equation*}
B_{3}(\ell )=4\pi \sum_{m=-\ell }^{\ell }\{|a_{\ell m}|^{2}-1\}\left[ \frac{4%
}{3\ell }-\frac{2m}{\ell ^{2}}+\frac{2m^{3}}{3\ell ^{4}}\right] +r_{3}(\ell
),
\end{equation*}%
\begin{equation*}
B_{4}(\ell )=4\pi \times 2\sum_{m=-\ell }^{\ell }\{|a_{\ell m}|^{2}-1\}\left[
\frac{m}{\ell ^{2}}-\frac{m^{3}}{\ell ^{4}}\right] +r_{4}(\ell ),
\end{equation*}%
\begin{equation*}
B_{5}(\ell )=4\pi \times \frac{1}{6}\sum_{m=-\ell }^{\ell }\{|a_{\ell
m}|^{2}-1\}\left[ \frac{1}{\ell }+\frac{8m^{3}}{\ell ^{4}}\right]
+r_{5}(\ell ),
\end{equation*}%
where $\sqrt{\mathbb{E}\left[ r_{i}(\ell )\right] ^{2}}=O(\ell ^{-1}),$ for
all $i=0,...,5.$
\end{proposition}

The proof of Proposition \ref{AB} is postponed to the Appendix \ref{proofAB}%
. We are now in the position to conclude the main proof of this Section.%
\newline

\begin{proof}[Proof of Theorem \protect\ref{proj2}]
A simple rewriting of the results from Proposition \ref{hk} yields 
\begin{equation*}
h_{35}(u;\ell )=\ell ^{2}\left\{ \frac{\Phi (-u)}{2\sqrt{2}\pi }+u\phi (u)%
\frac{u^{2}+1}{6\sqrt{2}\pi }\right\} +O(\ell ),
\end{equation*}%
and also%
\begin{equation*}
k_{1}(u;\ell )=k_{2}(u;\ell )=-\ell ^{2}\frac{u\phi (u)}{4\pi }+O(\ell ),
\end{equation*}%
\begin{equation*}
k_{3}(u;\ell )=\ell ^{2}\left\{ \frac{\Phi (-u)}{4\pi }+u\phi (u)\frac{%
u^{2}+1}{6\pi }\right\} +O(\ell ),
\end{equation*}%
\begin{equation*}
k_{4}(u;\ell )=-\ell ^{2}\frac{\Phi (-u)}{4\pi }+O(\ell ),\hspace{0.7cm}%
k_{5}(u;\ell )=\ell ^{2}u\phi (u)\frac{u^{2}+1}{12\pi }+O(\ell ),
\end{equation*}%
where the terms $O(\ell )$ are all uniform over $u.$ Now replacing the
expressions which were derived in Proposition \ref{AB}, we can hence write
down the projection on the second order Wiener chaos as follows:%
\begin{equation*}
\mathtt{Proj}[\chi (A_{u}(f_{\ell };\mathbb{S}^{2}))|2]
\end{equation*}%
\begin{eqnarray*}
&=&\ell ^{2}\left\{ \frac{\Phi (-u)}{2\sqrt{2}\pi }+u\phi (u)\frac{u^{2}+1}{6%
\sqrt{2}\pi }\right\} \left\{ 4\pi \frac{\sqrt{2}}{3}\sum_{m=-\ell }^{\ell
}\{|a_{\ell m}|^{2}-1\}\left[ -\frac{1}{\ell }+\frac{3\,m}{\ell ^{2}}-\frac{%
2\,m^{3}}{\ell ^{4}}\right] \right\} \\
&&-\frac{1}{2}\ell ^{2}\frac{u\phi (u)}{4\pi }\left\{ 4\pi \,\sum_{m=-\ell
}^{\ell }\{|a_{\ell m}|^{2}-1\}\left[ \frac{1}{\ell }-\frac{m}{\ell ^{2}}%
\right] \right\} \\
&&-\frac{1}{2}\ell ^{2}\frac{u\phi (u)}{4\pi }\left\{ 4\pi \,\sum_{m=-\ell
}^{\ell }\{|a_{\ell m}|^{2}-1\}\frac{m}{\ell ^{2}}\right\} \\
&&+\frac{1}{2}\ell ^{2}\left\{ \frac{\Phi (-u)}{4\pi }+u\phi (u)\frac{u^{2}+1%
}{6\pi }\right\} \left\{ 4\pi \,\sum_{m=-\ell }^{\ell }\{|a_{\ell m}|^{2}-1\}%
\left[ \frac{4}{3\ell }-\frac{2m}{\ell ^{2}}+\frac{2m^{3}}{3\ell ^{4}}\right]
\right\} \\
&&-\frac{1}{2}\ell ^{2}\frac{\Phi (-u)}{4\pi }\left\{ 4\pi \,\times
2\sum_{m=-\ell }^{\ell }\{|a_{\ell m}|^{2}-1\}\left[ \frac{m}{\ell ^{2}}-%
\frac{m^{3}}{\ell ^{4}}\right] \right\} \\
&&+\frac{1}{2}\ell ^{2}u\phi (u)\frac{u^{2}+1}{12\pi }\left\{ 4\pi \,\frac{1%
}{6}\sum_{m=-\ell }^{\ell }\{|a_{\ell m}|^{2}-1\}\left[ \frac{1}{\ell }+%
\frac{8m^{3}}{\ell ^{4}}\right] \right\} +R_{1}(\ell ),
\end{eqnarray*}%
where the remainder term $R_{1}(\ell )$ is such that $\sqrt{\mathbb{E}%
[R_{1}(\ell )]^{2}}=O(\ell ),$ again uniformly over $u.$ We now show that
all terms which include the Gaussian cumulative distribution function cancel
exactly; more precisely, performing some simple manipulations it is
immediate to note that 
\begin{equation*}
\ell ^{2}\frac{\Phi (-u)}{2\sqrt{2}\pi }A_{35}(\ell )+\frac{1}{2}\ell ^{2}%
\frac{\Phi (-u)}{4\pi }B_{3}(\ell )-\frac{1}{2}\ell ^{2}\frac{\Phi (-u)}{%
4\pi }B_{4}(\ell )
\end{equation*}%
\begin{eqnarray*}
&=&\ell ^{2}\frac{\Phi (-u)}{2\sqrt{2}\pi }\left\{ 4\pi \frac{\sqrt{2}}{3}%
\sum_{m=-\ell }^{\ell }\{|a_{\ell m}|^{2}-1\}\left[ -\frac{1}{\ell }+\frac{%
3\,m}{\ell ^{2}}-\frac{2\,m^{3}}{\ell ^{4}}\right] \right\} \\
&&+\frac{1}{2}\ell ^{2}\frac{\Phi (-u)}{4\pi }\left\{ 4\pi \,\sum_{m=-\ell
}^{\ell }\{|a_{\ell m}|^{2}-1\}\left[ \frac{4}{3\ell }-\frac{2m}{\ell ^{2}}+%
\frac{2m^{3}}{3\ell ^{4}}\right] \right\} \\
&&-\frac{1}{2}\ell ^{2}\frac{\Phi (-u)}{4\pi }\left\{ 4\pi \,\times
2\sum_{m=-\ell }^{\ell }\{|a_{\ell m}|^{2}-1\}\left[ \frac{m}{\ell ^{2}}-%
\frac{m^{3}}{\ell ^{4}}\right] \right\} +R_{2}(\ell )
\end{eqnarray*}%
\begin{eqnarray*}
&=&2\ell ^{2}\frac{\Phi (-u)}{3}\left\{ \sum_{m=-\ell }^{\ell }\{|a_{\ell
m}|^{2}-1\}\left[ -\frac{1}{\ell }+\frac{3\,m}{\ell ^{2}}-\frac{2\,m^{3}}{%
\ell ^{4}}\right] \right\} \\
&&+\frac{1}{2}\ell ^{2}\Phi (-u)\left\{ \,\sum_{m=-\ell }^{\ell }\{|a_{\ell
m}|^{2}-1\}\left[ \frac{4}{3\ell }-\frac{2m}{\ell ^{2}}+\frac{2m^{3}}{3\ell
^{4}}\right] \right\} \\
&&-\ell ^{2}\Phi (-u)\left\{ \sum_{m=-\ell }^{\ell }\{|a_{\ell m}|^{2}-1\} 
\left[ \frac{m}{\ell ^{2}}-\frac{m^{3}}{\ell ^{4}}\right] \right\}
+R_{2}(\ell )=R_{2}(\ell ),
\end{eqnarray*}%
where again the remainder term is uniformly bounded by $O(\ell )$ in the
mean-square norm. Rearranging the remaining terms, we thus obtain%
\begin{equation*}
\mathtt{Proj}[\chi (A_{u}(f_{\ell };\mathbb{S}^{2}))|2]
\end{equation*}%
\begin{equation*}
=\ell ^{2}u\phi (u)\frac{u^{2}+1}{6\sqrt{2}\pi }A_{35}(\ell )-\frac{1}{2}%
\ell ^{2}u\phi (u)\frac{1}{4\pi }\{B_{1}(\ell )+B_{2}(\ell )\}
\end{equation*}%
\begin{eqnarray*}
&=&\ell ^{2}u\phi (u)\frac{u^{2}+1}{6\sqrt{2}\pi }\left\{ 4\pi \frac{\sqrt{2}%
}{3}\sum_{m=-\ell }^{\ell }\{|a_{\ell m}|^{2}-1\}\left[ -\frac{1}{\ell }+%
\frac{3\,m}{\ell ^{2}}-\frac{2\,m^{3}}{\ell ^{4}}\right] \right\} \\
&&-\frac{1}{2}\ell ^{2}u\phi (u)\frac{1}{4\pi }\left\{ 4\pi \,\sum_{m=-\ell
}^{\ell }\{|a_{\ell m}|^{2}-1\}\left[ \frac{1}{\ell }-\frac{m}{\ell ^{2}}%
\right] \right\} \\
&&-\frac{1}{2}\ell ^{2}u\phi (u)\frac{1}{4\pi }\left\{ 4\pi \,\sum_{m=-\ell
}^{\ell }\{|a_{\ell m}|^{2}-1\}\frac{m}{\ell ^{2}}\right\} \\
&&+\frac{1}{2}\ell ^{2}\left\{ u\phi (u)\frac{u^{2}+1}{6\pi }\right\}
\left\{ 4\pi \,\sum_{m=-\ell }^{\ell }\{|a_{\ell m}|^{2}-1\}\left[ \frac{4}{%
3\ell }-\frac{2m}{\ell ^{2}}+\frac{2m^{3}}{3\ell ^{4}}\right] \right\} \\
&&+\frac{1}{2}\ell ^{2}u\phi (u)\frac{u^{2}+1}{12\pi }\left\{ 4\pi \,\frac{1%
}{6}\sum_{m=-\ell }^{\ell }\{|a_{\ell m}|^{2}-1\}\left[ \frac{1}{\ell }+%
\frac{8m^{3}}{\ell ^{4}}\right] \right\} +R(\ell )
\end{eqnarray*}%
\begin{equation*}
=\ell ^{2}u\phi (u)(u^{2}+1)\frac{2}{9}\left\{ \sum_{m=-\ell }^{\ell
}\{|a_{\ell m}|^{2}-1\}\left[ -\frac{1}{\ell }+\frac{3\,m}{\ell ^{2}}-\frac{%
2\,m^{3}}{\ell ^{4}}\right] \right\}
\end{equation*}%
\begin{equation*}
-\frac{1}{2}\ell ^{2}u\phi (u)\left\{ \,\sum_{m=-\ell }^{\ell }\{|a_{\ell
m}|^{2}-1\}\left[ \frac{1}{\ell }-\frac{m}{\ell ^{2}}\right] \right\} -\frac{%
1}{2}\ell ^{2}u\phi (u)\left\{ \,\sum_{m=-\ell }^{\ell }\{|a_{\ell
m}|^{2}-1\}\frac{m}{\ell ^{2}}\right\}
\end{equation*}%
\begin{equation*}
+\frac{1}{2}\ell ^{2}\left\{ u\phi (u)(u^{2}+1)\frac{2}{3}\right\} \left\{
\,\sum_{m=-\ell }^{\ell }\{|a_{\ell m}|^{2}-1\}\left[ \frac{4}{3\ell }-\frac{%
2m}{\ell ^{2}}+\frac{2m^{3}}{3\ell ^{4}}\right] \right\}
\end{equation*}%
\begin{equation*}
+\frac{1}{2}\ell ^{2}u\phi (u)\frac{u^{2}+1}{18}\left\{ \sum_{m=-\ell
}^{\ell }\{|a_{\ell m}|^{2}-1\}\left[ \frac{1}{\ell }+\frac{8m^{3}}{\ell ^{4}%
}\right] \right\} +R(\ell )
\end{equation*}%
\begin{equation*}
=\ell ^{2}u\phi (u)(u^{2}+1)\frac{2}{9}\left\{ -\frac{1}{\ell }\sum_{m=-\ell
}^{\ell }\{|a_{\ell m}|^{2}-1\}\right\} -\frac{1}{2}\ell ^{2}u\phi
(u)\left\{ \,\sum_{m=-\ell }^{\ell }\{|a_{\ell m}|^{2}-1\}\frac{1}{\ell }%
\right\}
\end{equation*}%
\begin{equation*}
+\frac{1}{2}\ell ^{2}\left\{ u\phi (u)(u^{2}+1)\frac{2}{3}\right\} \left\{ \,%
\frac{4}{3\ell }\sum_{m=-\ell }^{\ell }\{|a_{\ell m}|^{2}-1\}\right\} +\frac{%
1}{2}\ell ^{2}u\phi (u)\frac{u^{2}+1}{18}\left\{ \frac{1}{\ell }%
\sum_{m=-\ell }^{\ell }\{|a_{\ell m}|^{2}-1\}\right\} +R(\ell )
\end{equation*}%
\begin{equation*}
=\ell u\phi (u)\frac{u^{2}-1}{4}\,\sum_{m=-\ell }^{\ell }\{|a_{\ell
m}|^{2}-1\}+R(\ell ),
\end{equation*}%
where $\sqrt{\mathbb{E}R^{2}(\ell )}=O(\ell ),$ as claimed.
\end{proof}

\section{Variance and Quantitative Central Limit Theorem}

\label{Variance_CLT}

In this section we prove Theorem \ref{clt}. Our first result is the
following.

\begin{lemma}
\label{Variance2ndchaos} As $\ell \rightarrow \infty ,$ for all $u\neq 0$ we
have that%
\begin{equation*}
\lim_{\ell \rightarrow \infty }\frac{\mathrm{Var}[\mathtt{Proj}[\chi
(A_{u}(f_{\ell };\mathbb{S}^{2}))|2]]}{\mathrm{Var}[\chi (A_{u}(f_{\ell };%
\mathbb{S}^{2}))]}=1+O\left( \frac{\log \ell }{\ell }\right) \text{ .}
\end{equation*}
\end{lemma}

\begin{proof}
In \cite{CMW}, \cite{CW} it is shown that, for all $u\neq 0$ 
\begin{equation*}
\mathrm{Var}[\chi (A_{u}(f_{\ell };\mathbb{S}^{2}))]=\frac{1}{4}\ell
^{3}\left\{ u\phi (u)(u^{2}-1)\right\} ^{2}+O\left( \ell ^{2}\log \ell
\right) \text{ ,}
\end{equation*}%
the error term being uniform over $u.$ In view of the form of $\mathtt{Proj}%
[\chi (A_{u}(f_{\ell };\mathbb{S}^{2}))|2],$ we need only consider the
asymptotic variance of $\,\sum_{m=-\ell }^{\ell }\{|a_{\ell m}|^{2}-1\}$;
the details are trivial, but we report them for completeness. Recall first
that 
\begin{equation*}
|a_{\ell m}|^{2}=\left\{ \mathrm{Re}(a_{\ell m})\right\} ^{2}+\left\{ 
\mathrm{Im}(a_{\ell m})\right\} ^{2}=|a_{\ell ,-m}|^{2},
\end{equation*}%
where $\mathrm{Re}(a_{\ell m}),\mathrm{Im}(a_{\ell m})$ are zero-mean,
independent Gaussian variables with variance $\frac{1}{2}$; on the other
hand, $a_{\ell 0}$ follows a standard $N(0,1)$ Gaussian distribution. We can
hence write%
\begin{equation*}
\sum_{m=-\ell }^{\ell }\{|a_{\ell m}|^{2}-1\}=\{|a_{\ell
0}|^{2}-1\}+2\sum_{m=1}^{\ell }\{|a_{\ell m}|^{2}-1\}
\end{equation*}%
\begin{eqnarray*}
&=&\{|a_{\ell 0}|^{2}-1\}+2\sum_{m=1}^{\ell }\{\mathrm{Re}|a_{\ell m}|^{2}-%
\frac{1}{2}\}+2\sum_{m=1}^{\ell }\{\mathrm{Im}|a_{\ell m}|^{2}-\frac{1}{2}\}
\\
&=&\{|a_{\ell 0}|^{2}-1\}+\sum_{m=1}^{\ell }\{\mathrm{Re}|\sqrt{2}a_{\ell
m}|^{2}-1\}+\sum_{m=1}^{\ell }\{\mathrm{Im}|\sqrt{2}a_{\ell m}|^{2}-1\}.
\end{eqnarray*}%
Now note that $|a_{\ell 0}|^{2}$, $\mathrm{Re}|\sqrt{2}a_{\ell m}|^{2}$, $%
\mathrm{Im}|\sqrt{2}a_{\ell m}|^{2}$, $m=1,...,\ell $ are a set of $2\ell +1$
independent variables distributed according to a $\chi _{1}^{2}$ with one
degree of freedom; it follows immediately that%
\begin{equation*}
\mathrm{Var}\left[ \sum_{m=-\ell }^{\ell }\{|a_{\ell m}|^{2}-1\}\right]
=2(2\ell +1).
\end{equation*}%
Thus 
\begin{eqnarray*}
\lim_{\ell \rightarrow \infty }\frac{\mathrm{Var}\left[ \mathtt{Proj}[\chi
(A_{u}(f_{\ell };\mathbb{S}^{2}))|2]\right] }{\frac{1}{4}\ell ^{3}\left\{
u\phi (u)(u^{2}-1)\right\} ^{2}} &=&\lim_{\ell \rightarrow \infty }\frac{%
\mathrm{Var}\left[ \frac{1}{4}\ell u\phi (u)(u^{2}-1)\,\sum_{m=-\ell }^{\ell
}\{|a_{\ell m}|^{2}-1\}\right] }{\frac{1}{4}\ell ^{3}\left\{ u\phi
(u)(u^{2}-1)\right\} ^{2}} \\
&=&\frac{1}{4}\lim_{\ell \rightarrow \infty }\frac{\mathrm{Var}\,\left[
\sum_{m=-\ell }^{\ell }\{|a_{\ell m}|^{2}-1\}\right] }{\ell }=1,
\end{eqnarray*}%
and the result we claimed follows immediately.
\end{proof}

\bigskip

\begin{proof}[Proof of Theorem \protect\ref{clt}]
We recall that the Wasserstein distance between random variables $X,Y$ is
defined by 
\begin{equation*}
d_{W}(X,Y):=\sup_{h\in Lip(1)}\left\vert \mathbb{E}h(X)-\mathbb{E}%
h(Y)\right\vert ;
\end{equation*}%
also, $d_{W}(X,Y)\leq \sqrt{\mathbb{E}\left\vert X-Y\right\vert ^{2}},$ i.e.
Wasserstein distance is always bounded by the $L^{2}$-metric, see \cite%
{noupebook} for further characterizations and details. By the triangle
inequality, we have%
\begin{equation*}
d_{W}\left( \frac{\chi (A_{u}(f_{\ell };\mathbb{S}^{2}))-\mathbb{E}[\chi
(A_{u}(f_{\ell };\mathbb{S}^{2}))]}{\sqrt{\mathrm{Var}[\chi (A_{u}(f_{\ell };%
\mathbb{S}^{2}))]}},Z\right)
\end{equation*}%
\begin{eqnarray*}
&\leq &d_{W}\left( \frac{\chi (A_{u}(f_{\ell };\mathbb{S}^{2}))-\mathbb{E}%
[\chi (A_{u}(f_{\ell };\mathbb{S}^{2}))]}{\sqrt{\mathrm{Var}[\chi
(A_{u}(f_{\ell };\mathbb{S}^{2}))]}},\frac{\mathtt{Proj}[\chi (A_{u}(f_{\ell
};\mathbb{S}^{2}))|2]}{\sqrt{\mathrm{Var}[\chi (A_{u}(f_{\ell };\mathbb{S}%
^{2}))]}}\right) \\
&&+d_{W}\left( \frac{\mathtt{Proj}[\chi (A_{u}(f_{\ell };\mathbb{S}^{2}))|2]%
}{\sqrt{\mathrm{Var}[\chi (A_{u}(f_{\ell };\mathbb{S}^{2}))]}},Z\right) 
\text{ ,}
\end{eqnarray*}%
and hence%
\begin{eqnarray*}
&&d_{W}\left( \frac{\chi (A_{u}(f_{\ell };\mathbb{S}^{2}))-\mathbb{E}[\chi
(A_{u}(f_{\ell };\mathbb{S}^{2}))]}{\sqrt{\mathrm{Var}[\chi (A_{u}(f_{\ell };%
\mathbb{S}^{2}))]}},Z\right) \\
&=&d_{W}\left( \frac{\mathtt{Proj}[\chi (A_{u}(f_{\ell };\mathbb{S}^{2}))|2]%
}{\sqrt{\mathrm{Var}[\chi (A_{u}(f_{\ell };\mathbb{S}^{2}))]}},Z\right) +O(%
\sqrt{\frac{\log \ell }{\ell }})\text{ ,}
\end{eqnarray*}%
because%
\begin{equation*}
\mathbb{E}\left\{ \frac{\chi (A_{u}(f_{\ell };\mathbb{S}^{2}))-\mathbb{E}%
[\chi (A_{u}(f_{\ell };\mathbb{S}^{2}))]-\mathtt{Proj}[\chi (A_{u}(f_{\ell };%
\mathbb{S}^{2}))|2]}{\sqrt{\mathrm{Var}[\chi (A_{u}(f_{\ell };\mathbb{S}%
^{2}))]}}\right\} ^{2}=O(\frac{\log \ell }{\ell }),
\end{equation*}%
uniformly over $u.$ By a similar argument%
\begin{equation*}
d_{W}\left( \frac{\mathtt{Proj}[\chi (A_{u}(f_{\ell };\mathbb{S}^{2}))|2]}{%
\sqrt{\mathrm{Var}[\chi (A_{u}(f_{\ell };\mathbb{S}^{2}))]}},Z\right)
=d_{W}(F_{\ell }(u);Z)+O(\sqrt{\frac{1}{\ell }})\text{ ,}
\end{equation*}%
where we wrote for notational simplicity%
\begin{equation*}
F_{\ell }(u):=\frac{\frac{\lambda _{\ell }}{2}\left\{ H_{1}(u)H_{2}(u)\phi
(u)\right\} \frac{1}{2\ell +1}\,\sum_{m=-\ell }^{\ell }\{|a_{\ell m}|^{2}-1\}%
}{\sqrt{\mathrm{Var}[\chi (A_{u}(f_{\ell };\mathbb{S}^{2}))]}}\text{ ;}
\end{equation*}%
from Corollary 5.2.10 in \cite{noupebook} we have%
\begin{eqnarray*}
d_{W}\left( F_{\ell }(u),Z\right) &\leq &\sqrt{\frac{2\left( \mathbb{E}%
F_{\ell }^{4}(u)-3\left[ \mathbb{E}F_{\ell }^{2}(u)\right] ^{2}\right) }{%
3\pi \left[ \mathbb{E}F_{\ell }^{2}(u)\right] ^{2}}}+\sqrt{\frac{\sqrt{\frac{%
2}{\pi }}\left( \mathbb{E}F_{\ell }^{2}(u)-1\right) }{\mathbb{E}F_{\ell
}^{2}(u)\vee 1}} \\
&=&\sqrt{\frac{2\left( \mathbb{E}F_{\ell }^{4}(u)-3\left[ \mathbb{E}F_{\ell
}^{2}(u)\right] ^{2}\right) }{3\pi \left[ \mathbb{E}F_{\ell }^{2}(u)\right]
^{2}}}+O\left( \sqrt{\frac{\log \ell }{\ell }}\right) ,
\end{eqnarray*}%
in view of Lemma \ref{Variance2ndchaos}. To complete the proof, it suffices
to notice that for every fixed $u,$ $\mathbb{E}F_{\ell }^{4}(u)-3\left[ 
\mathbb{E}F_{\ell }^{2}(u)\right] ^{2}$ is the fourth-order cumulant of the
sample average of $2\ell +1$ independent random variables with finite
moments of all order; it is then a standard exercise to show that this
quantity is $O(\ell ^{-1}),$ which completes the proof$.$
\end{proof}

\begin{remark}
The Theorem can be generalized to joint convergence for every fixed set of
threshold levels $(u_{1},...,u_{p}),$ $p\in \mathbb{N}$; details are trivial
and hence omitted. A more interesting possibility would be to investigate a
Functional Central Limit Theorem over $u;$ this extensions seems possible,
but we do not consider it here for brevity's sake.
\end{remark}

\section{Appendix A: Proof of Proposition \protect\ref{hk}}

\label{proofhk}

\noindent Let $Y$ be a standard random variable; for the projection
coefficients of the Dirac's delta function, (which are given for instance in 
\cite{noupebook}, Chapter 1, see also \cite{MPRW2015}), we introduce the
following notation: 
\begin{equation*}
\varphi _{a}(\ell )=\lim_{\varepsilon \rightarrow 0}\mathbb{E}%
[H_{a}(Y)\delta _{\varepsilon }(\lambda _{1}\,Y)],\hspace{1cm}a=0,1,2.
\end{equation*}%
We also use $\theta _{ab}$ to denote projection coefficients involving two
random variables $Y_{a},Y_{b}$ and $\psi _{abcd}(u)$ to denote those
coefficients that involve four, i.e., we set 
\begin{equation}
\theta _{ab}(u)=\mathbb{E}\left[ Y_{a}Y_{b}1\hspace{-0.27em}\mbox{\rm l}%
_{\left\{ \frac{\lambda _{2}+\lambda _{3}}{\lambda }Y_{3}+\frac{\lambda _{5}%
}{\lambda }Y_{5}\leq -u\right\} }\right] ,\hspace{1cm}a,b=3,4,5,
\label{thetadef}
\end{equation}%
and 
\begin{equation}
\psi _{abcd}(u)=\mathbb{E}\left[ Y_{a}Y_{b}Y_{c}Y_{d}1\hspace{-0.27em}%
\mbox{\rm l}_{\left\{ \frac{\lambda _{2}+\lambda _{3}}{\lambda }Y_{3}+\frac{%
\lambda _{5}}{\lambda }Y_{5}\leq -u\right\} }\right] ,\hspace{1cm}%
a,b,c,d=3,4,5.  \label{psi}
\end{equation}%
The exact behaviour of these coefficients as a function of the level $u$ is
given in the three Lemmas to follow.

\begin{lemma}
\label{varphi} We have 
\begin{equation*}
\varphi_{a}(\ell )=%
\begin{cases}
\frac{1}{\sqrt{2\pi }\lambda _{1}}, & a=0, \\ 
0, & a=1, \\ 
-\frac{1}{\sqrt{2\pi }\lambda _{1}}, & a=2.%
\end{cases}%
\end{equation*}
\end{lemma}

\begin{proof}
The result follows from the straightforward computation%
\begin{align*}
\varphi _{0}(\ell )& =\lim_{\varepsilon \rightarrow 0}\frac{1}{2\varepsilon }%
\int_{-\infty }^{\infty }H_{0}(y)\mathbb{I}_{[-\varepsilon ,\varepsilon
]}(\lambda _{1}y)\phi (y)dy \\
& =\lim_{\varepsilon \rightarrow 0}\frac{1}{2\varepsilon }\int_{-\varepsilon
/\lambda _{1}}^{\varepsilon /\lambda _{1}}\phi (y)dy=\frac{1}{\sqrt{2\pi }%
\lambda _{1}}\text{ };
\end{align*}%
the proof for $a=1,2$ is analogous.
\end{proof}

\bigskip

\noindent In what follows, to simplify the notation, we set $\alpha _{\ell }=%
\frac{\lambda _{2}+\lambda _{3}}{\lambda _{\ell }}$, $\beta _{\ell }=\frac{%
\lambda _{5}}{\lambda _{\ell }}$. Note that $\alpha _{\ell }^{2}+\beta
_{\ell }^{2}=1$ and $\alpha _{\ell }^{2}=\frac{2\lambda _{\ell }}{3\lambda
_{\ell }-2}$; we recall once again that we use $\phi (.)$ and $\Phi (.)$ to
denote as usual the density and distribution function of a standard Gaussian
random variable. Our next results are concerned with the analytic
expressions for the function $\theta _{ab}$.

\begin{lemma}
\label{theta}We have that 
\begin{equation*}
\theta _{33}(u)=\Phi (-u)+\frac{2\lambda _{\ell }}{3\lambda _{\ell }-2}%
u\,\phi (u),\;\theta _{35}(u)=\sqrt{2}\frac{\sqrt{\lambda _{\ell }}\sqrt{%
\lambda _{\ell }-2}}{3\lambda _{\ell }-2}\,u\,\phi (u),\;\theta
_{44}(u)=\Phi (-u).
\end{equation*}
\end{lemma}

\begin{proof}
The proof is a simple exercise in the computation of moments and
convolutions of normal variables. More precisely, let $X$, $Y$ and $Z$ be
three independent standard Gaussian random variables; in view of Lemma \ref%
{aa2}, we have 
\begin{align*}
\theta _{33}(u)&=\mathbb{E}\left[ Y^{2}1\hspace{-0.27em}\mbox{\rm l}%
_{\left\{ \alpha _{\ell }Y+\beta _{\ell }X\leq -u\right\} }\right]
=\int_{-\infty }^{\infty }y^{2}\phi (y)\Phi \left( \frac{-u-\alpha _{\ell }y%
}{\beta _{\ell }}\right) dy \\
&=\Phi (-u)+\alpha _{\ell }^{2}\,u\,\phi (-u).
\end{align*}%
Moreover 
\begin{equation*}
\theta _{35}(u)=\mathbb{E}\left[ XY1\hspace{-0.27em}\mbox{\rm l}_{\left\{
\alpha _{\ell }Y+\beta _{\ell }X\leq -u\right\} }\right] =\int_{-\infty
}^{\infty }y\phi (y)dy\int_{-\infty }^{\frac{-u-\alpha _{\ell }y}{\beta
_{\ell }}}x\phi (x)dx
\end{equation*}%
\begin{equation*}
=-\int_{-\infty }^{\infty }y\phi (y)\phi \left( \frac{-u-\alpha _{\ell }y}{%
\beta _{\ell }}\right) dy=\alpha _{\ell }\,\beta _{\ell }\,u\,\phi (-u),
\end{equation*}%
and finally by applying Lemma \ref{aa1} we obtain 
\begin{equation*}
\theta _{44}(u)=\mathbb{E}\left[ Z^{2}1\hspace{-0.27em}\mbox{\rm l}_{\left\{
\alpha _{\ell }Y+\beta _{\ell }X\leq -u\right\} }\right] =\int_{-\infty
}^{\infty }\phi (y)\Phi \left( \frac{-u-\alpha _{\ell }y}{\beta _{\ell }}%
\right) dy=\Phi (-u).
\end{equation*}
\end{proof}

The computation of expected values involving four moments is clearly more
challenging and is detailed in the Lemma below.

\begin{lemma}
a) The expression for the coefficients involving only $Y_{3}$ or $Y_{4}$ is
equal to%
\begin{equation*}
\psi _{3333}(u)=3\Phi (-u)+4\lambda _{\ell }\frac{\lambda _{\ell }(u^{2}+6)-6%
}{(3\lambda _{\ell }-2)^{2}}u\phi (u),\;\;\psi _{4444}(u)=3\Phi (-u).
\end{equation*}%
b) The expression for coefficients involving cross products of $Y_{3}$ and $%
Y_{5}$ are equal to%
\begin{align*}
\psi _{3355}(u)& =\Phi (-u)+\frac{4+2u^{2}\lambda _{\ell }(\lambda _{\ell
}-2)+3\lambda _{\ell }^{2}}{(3\lambda _{\ell }-2)^{2}}u\phi (u), \\
\psi _{3555}(u)& =\sqrt{2}(\lambda _{\ell }u^{2}-2u^{2}+6\lambda _{\ell })%
\frac{\sqrt{\lambda _{\ell }}\sqrt{\lambda _{\ell }-2}}{(3\lambda _{\ell
}-2)^{2}}u\phi (u), \\
\psi _{3335}(u)& =\sqrt{2}(2\lambda _{\ell }u^{2}+3\lambda _{\ell }-6)\frac{%
\sqrt{\lambda _{\ell }}\sqrt{\lambda _{\ell }-2}}{(3\lambda _{\ell }-2)^{2}}%
u\phi (u).
\end{align*}%
c) The expression for coefficients involving cross-products with $Y_{4}$ are
as follows:%
\begin{align*}
\psi _{3344}(u)& =\Phi (-u)+\frac{2\lambda _{\ell }}{3\lambda _{\ell }-2}%
u\phi (u), \\
\psi _{4455}(u)& =\Phi (-u)+\frac{\lambda _{\ell }-2}{3\lambda _{\ell }-2}%
u\phi (u), \\
\psi _{3445}(u)& =\sqrt{2}\frac{\sqrt{\lambda _{\ell }}\sqrt{\lambda _{\ell
}-2}}{3\lambda _{\ell }-2}u\phi (u).
\end{align*}%
d) The following remaining terms are identically zero:%
\begin{equation*}
\psi _{3334}(u)=\psi _{3345}(u)=\psi _{3444}(u)=\psi _{3455}(u)=\psi
_{4445}(u)=0.
\end{equation*}
\end{lemma}

\begin{proof}
Again, the proof is a rather straightforward, albeit long and tedious,
exercise in the computation of Gaussian moments and convolutions; for
notational simplicity, in the sequel we shall use $X$, $Y$ and $Z$ to denote
three independent standard Gaussian random variables. To prove a), by
applying Lemma \ref{aa3} we have 
\begin{align*}
\psi _{3333}(u)& =\mathbb{E}\left[ Y^{4}1\hspace{-0.27em}\mbox{\rm l}%
_{\left\{ \alpha _{\ell }Y+\beta _{\ell }X\leq -u\right\} }\right]
=\int_{-\infty }^{\infty }y^{4}\phi (y)\Phi \Big(\frac{-u-\alpha _{\ell }y}{%
\beta _{\ell }}\Big)dy \\
& =3\Phi (-u)+3\alpha _{\ell }^{2}u\phi (-u)+3\alpha _{\ell }^{4}\beta
_{\ell }^{2}u\phi (-u)+3\beta _{\ell }^{4}\alpha _{\ell }^{2}u\phi
(-u)+\alpha _{\ell }^{4}u^{3}\phi (-u).
\end{align*}%
Likewise, from Lemma \ref{aa1}, 
\begin{equation*}
\psi _{4444}(u)=\mathbb{E}\left[ Z^{4}1\hspace{-0.27em}\mbox{\rm l}_{\left\{
\alpha _{\ell }Y+\beta _{\ell }X\leq -u\right\} }\right] =3\,\mathbb{E}\left[
1\hspace{-0.27em}\mbox{\rm l}_{\left\{ \alpha _{\ell }Y+\beta _{\ell }X\leq
-u\right\} }\right] =3\Phi (-u).
\end{equation*}%
To prove b), we start by observing that 
\begin{equation*}
\int_{-\infty }^{q}x^{2}\phi (x)dx=\Phi (q)-q\,\phi (q),
\end{equation*}%
and from Lemma \ref{aa2}, we obtain 
\begin{align*}
\psi _{3355}(u)& =\mathbb{E}\left[ Y^{2}X^{2}1\hspace{-0.27em}\mbox{\rm l}%
_{\left\{ \alpha _{\ell }Y+\beta _{\ell }X\leq -u\right\} }\right]
=\int_{-\infty }^{\infty }y^{2}\phi (y)dy\int_{-\infty }^{\frac{-u-\alpha
_{\ell }y}{\beta _{\ell }}}x^{2}\phi (x)dx \\
& =\int_{-\infty }^{\infty }y^{2}\phi (y)\Phi \left( \frac{-u-\alpha _{\ell
}y}{\beta _{\ell }}\right) dy-\int_{-\infty }^{\infty }y^{2}\phi (y)\left( 
\frac{-u-\alpha _{\ell }y}{\beta _{\ell }}\right) \phi \left( \frac{%
-u-\alpha _{\ell }y}{\beta _{\ell }}\right) dy \\
& =\Phi (-u)+\alpha _{\ell }^{2}\,u\,\phi (-u)+\beta _{\ell
}^{2}\,u\,(-2\alpha _{\ell }^{4}+\beta _{\ell }^{4}-\alpha _{\ell }^{2}\beta
_{\ell }^{2}+\alpha _{\ell }^{2}u^{2})\phi (-u).
\end{align*}%
The proof of all remaining terms is very similar. For instance,%
\begin{align*}
\psi _{3555}(u)& =\mathbb{E}\left[ YX^{3}1\hspace{-0.27em}\mbox{\rm l}%
_{\left\{ \alpha _{\ell }Y+\beta _{\ell }X\right\} \leq -u}\right]
=\int_{-\infty }^{\infty }y\phi (y)dy\int_{-\infty }^{\frac{-u-\alpha _{\ell
}y}{\beta _{\ell }}}x^{3}\phi (x)dx \\
& =-\int_{-\infty }^{\infty }y\phi (y)\phi \left( \frac{-u-\alpha _{\ell }y}{%
\beta _{\ell }}\right) \left\{ \left( \frac{-u-\alpha _{\ell }y}{\beta
_{\ell }}\right) ^{2}+2\right\} dy \\
& =\alpha _{\ell }\beta _{\ell }u\,(3\alpha _{\ell }^{2}+\beta _{\ell
}^{2}u^{2})\phi (-u),
\end{align*}%
and likewise%
\begin{equation*}
\psi _{3335}(u)=\mathbb{E}\left[ Y^{3}X1\hspace{-0.27em}\mbox{\rm l}%
_{\left\{ \alpha _{\ell }Y+\beta _{\ell }X\leq -u\right\} }\right] =\alpha
_{\ell }\beta _{\ell }u\,(3\beta _{\ell }^{2}+\alpha _{\ell }^{2}u^{2})\phi
(-u).
\end{equation*}%
To prove c) it is enough to note that%
\begin{equation*}
\psi _{3344}(u)=\mathbb{E}\left[ Z^{2}Y^{2}1\hspace{-0.27em}\mbox{\rm l}%
_{\left\{ \alpha _{\ell }Y+\beta _{\ell }X\leq -u\right\} }\right] =\mathbb{E%
}\left[ Y^{2}1\hspace{-0.27em}\mbox{\rm l}_{\left\{ \alpha _{\ell }Y+\beta
_{\ell }X\leq -u\right\} }\right] =\theta _{33}(u),
\end{equation*}%
\begin{align*}
\psi _{4455}(u)& =\mathbb{E}\left[ Z^{2}X^{2}1\hspace{-0.27em}\mbox{\rm l}%
_{\left\{ \alpha _{\ell }Y+\beta _{\ell }X\leq -u\right\} }\right] =\mathbb{E%
}\left[ X^{2}1\hspace{-0.27em}\mbox{\rm l}_{\left\{ \alpha _{\ell }Y+\beta
_{\ell }X\leq -u\right\} }\right] =\theta _{55}(u) \\
& =\Phi (-u)+\beta _{\ell }^{2}u\phi (-u),
\end{align*}%
and 
\begin{equation*}
\psi _{3445}(u)=\mathbb{E}\left[ Z^{2}XY1\hspace{-0.27em}\mbox{\rm l}%
_{\left\{ \alpha _{\ell }Y+\beta _{\ell }X\leq -u\right\} }\right] =\mathbb{E%
}\left[ XY1\hspace{-0.27em}\mbox{\rm l}_{\left\{ \alpha _{\ell }Y+\beta
_{\ell }X\leq -u\right\} }\right] =\theta _{35}(u).
\end{equation*}%
To prove d), i.e., the fact that $\psi _{3334}(u)$, $\psi _{3345}(u)$, $\psi
_{3444}(u)$, $\psi _{3455}(u)$ and $\psi _{4445}(u)$ are identically equal
to zero, it is enough to note that they are all of the form 
\begin{equation*}
\mathbb{E}\left[ Z^{p}X^{q}Y^{r}1\hspace{-0.27em}\mbox{\rm l}_{\left\{
\alpha _{\ell }Y+\beta _{\ell }X\leq -u\right\} }\right]
\end{equation*}%
where $p=1,3$ is odd.
\end{proof}

Some auxiliary computations which we exploited in the proof are collected in
Lemmas \ref{aa1}-\ref{aa3} below.

\begin{lemma}
\label{aa1} For all values of $\alpha _{\ell }$, $\beta _{\ell }$ and $u$,
the following identity holds: 
\begin{equation*}
\int_{-\infty }^{\infty }\phi (y)\Phi \left( \frac{-u-\alpha _{\ell }y}{%
\beta _{\ell }}\right) dy=\Phi (-u).
\end{equation*}
\end{lemma}

\begin{proof}
Recall first that 
\begin{equation*}
\Phi (x)=\frac{1}{2}+\frac{1}{2}\mathrm{erf}\left( \frac{x}{\sqrt{2}}%
\right), \hspace{1cm}\mathrm{erf}(q)=\frac{2}{\sqrt{\pi }}%
\int_{0}^{q}e^{-t^{2}}dt,
\end{equation*}%
whence 
\begin{align*}
\int_{-\infty }^{\infty }\phi (y)\Phi \left( \frac{-u-\alpha _{\ell }y}{%
\beta _{\ell }}\right) dy&=\frac{1}{2}+\frac{1}{2}\int_{-\infty }^{\infty
}\phi (y)\mathrm{erf}\left( \frac{-u-\alpha _{\ell }y}{\sqrt{2}\beta _{\ell }%
}\right) dy \\
&=\frac{1}{2}+\frac{1}{2}\mathrm{erf}\left( \frac{-u}{\sqrt{2}}\right) =\Phi
(-u),
\end{align*}%
by recalling that $\alpha _{\ell }^{2}+\beta _{\ell }^{2}=1$.
\end{proof}

\begin{lemma}
\label{aa2}For all values of $\alpha _{\ell }$, $\beta _{\ell }$ and $u$, we
have that: 
\begin{equation*}
\int_{-\infty }^{\infty }y^{2}\phi (y)\Phi \left( \frac{-u-\alpha _{\ell }y}{%
\beta _{\ell }}\right) dy=\Phi (-u)+\alpha _{\ell }^{2}\,u\,\phi (-u).
\end{equation*}
\end{lemma}

\begin{proof}
Note that 
\begin{equation}  \label{cinque}
y^{2}\phi (y)\Phi \left( \frac{-u-\alpha _{\ell }y}{\beta _{\ell }}\right)
=-y\left( \frac{d}{dy}\phi (y)\right) \Phi \left( \frac{-u-\alpha _{\ell }y}{%
\beta _{\ell }}\right),
\end{equation}
integrating by parts we have 
\begin{equation*}
\int_{-\infty }^{\infty }y^{2}\phi (y)\Phi \left( \frac{-u-\alpha _{\ell }y}{%
\beta _{\ell }}\right) dy=\int_{-\infty }^{\infty }\phi (y)\frac{d}{dy}%
\left\{ y\;\Phi \left( \frac{-u-\alpha _{\ell }y}{\beta _{\ell }}\right)
\right\} dy
\end{equation*}%
and since 
\begin{equation}  \label{unalabel}
\frac{d}{dy}\Phi \left( \frac{-u-\alpha _{\ell }y}{\beta _{\ell }}\right) =-%
\frac{\alpha _{\ell }}{\beta _{\ell }}\phi \left( \frac{-u-\alpha _{\ell }y}{%
\beta _{\ell }}\right)
\end{equation}
we obtain form Lemma \ref{aa1} and \ref{unalabel} that 
\begin{equation*}
\int_{-\infty }^{\infty }\phi (y)\frac{d}{dy}\left\{ y\;\Phi \left( \frac{%
-u-\alpha _{\ell }y}{\beta _{\ell }}\right) \right\} dy=\Phi (-u)-\frac{%
\alpha _{\ell }}{\beta _{\ell }}\int_{-\infty }^{\infty }y\,\phi (y)\,\phi
\left( \frac{-u-\alpha _{\ell }y}{\beta _{\ell }}\right) dy.
\end{equation*}%
The statement follows by observing that 
\begin{equation*}
\int_{-\infty }^{\infty }y\,\phi (y)\,\phi \left( \frac{-u-\alpha _{\ell }y}{%
\beta _{\ell }}\right) dy=-\alpha _{\ell }\beta _{\ell }u\,\phi (-u).
\end{equation*}
\end{proof}

\begin{lemma}
\label{aa3}For all values of $\alpha _{\ell }$, $\beta _{\ell }$ and $u$, it
holds that: 
\begin{align*}
&\int_{-\infty }^{\infty }y^{4}\phi (y)\,\Phi \left( \frac{-u-\alpha _{\ell
}y}{\beta _{\ell }}\right) dy \\
&=3\Phi (-u)+3\alpha _{\ell }^{2}u\phi (-u)+3\alpha _{\ell
}^{4}\beta_{\ell}^{2}u\phi (-u)+3\beta_{\ell}^{4}\alpha _{\ell }^{2}u\phi
(-u)+\alpha _{\ell }^{4}u^{3}\phi (-u).
\end{align*}
\end{lemma}

\begin{proof}
As in (\ref{cinque}) we write 
\begin{equation*}
y^{4}\phi (y)\,\Phi \left( \frac{-u-\alpha _{\ell }y}{\beta _{\ell }}\right)
=-y^{3}\left( \frac{d}{dy}\phi (y)\right) \Phi \left( \frac{-u-\alpha _{\ell
}y}{\beta _{\ell }}\right) ,
\end{equation*}%
so that integrating by parts we obtain 
\begin{align*}
& \int_{-\infty }^{\infty }y^{4}\phi (y)\,\Phi \left( \frac{-u-\alpha _{\ell
}y}{\beta _{\ell }}\right) dy=\int_{-\infty }^{\infty }\phi (y)\,\frac{d}{dy}%
\left\{ y^{3}\Phi \left( \frac{-u-\alpha _{\ell }y}{\beta _{\ell }}\right)
\right\} dy \\
& =3\int_{-\infty }^{\infty }y^{2}\;\phi (y)\,\Phi \left( \frac{-u-\alpha
_{\ell }y}{\beta _{\ell }}\right) dy+\int_{-\infty }^{\infty }y^{3}\,\phi
(y)\,\frac{d}{dy}\Phi \left( \frac{-u-\alpha _{\ell }y}{\beta _{\ell }}%
\right) dy,
\end{align*}%
the statement follows immediately by applying Lemma \ref{aa2} and by
observing that 
\begin{align*}
\int_{-\infty }^{\infty }y^{3}\,\phi (y)\,\frac{d}{dy}\Phi \left( \frac{%
-u-\alpha _{\ell }y}{\beta _{\ell }}\right) dy& =-\frac{\alpha _{\ell }}{%
\beta _{\ell }}\int_{-\infty }^{\infty }y^{3}\,\phi (y)\,\phi \left( \frac{%
-u-\alpha _{\ell }y}{\beta _{\ell }}\right) dy \\
& =\alpha _{\ell }^{2}\,u\,(3\beta _{\ell }^{4}+3\alpha _{\ell }^{2}\beta
_{\ell }^{2}+\alpha _{\ell }^{2}u^{2})\phi (-u).
\end{align*}%
\textbf{End of the proof of Proposition \ref{hk}.} We are now in the
position to complete the proof of the Proposition. First note that, in view
of Lemma \ref{varphi}, we immediately have $h_{1j}(u;\ell )=0$ for all $j>1$
and $h_{2j}(u;\ell )=0$ for all $j>2$ since $\varphi _{1}(\ell )=0$.
Moreover, some standard algebraic computations yield

\begin{equation*}
h_{34}(u;\ell )=[\lambda _{3}\lambda _{5}\;\psi _{3345}(u)+\lambda
_{2}\lambda _{3}\;\psi _{3334}(u)-\lambda _{4}^{2}\;\psi _{3444}(u)]\varphi
_{0}^{2}(\ell )=0
\end{equation*}

\begin{align*}
h_{35}(u;\ell )& =[\lambda _{3}\lambda _{5}\;\psi _{3355}(u)+\lambda
_{2}\lambda _{3}\;\psi _{3335}(u)-\lambda _{4}^{2}\;\psi _{3445}(u)]\varphi
_{0}^{2}(\ell ) \\
& =\sqrt{\lambda _{\ell }}\sqrt{\lambda _{\ell }-2}\frac{\Phi (-u)(3\lambda
_{\ell }-2)+u\phi (u)[2+\lambda _{\ell }(u^{2}+1)]}{2\sqrt{2}\pi (3\lambda
_{\ell }-2)},
\end{align*}

\begin{equation*}
h_{45}(u;\ell )=[\lambda _{3}\lambda _{5}\;\psi _{3455}(u)+\lambda
_{2}\lambda _{3}\;\psi _{3345}(u)-\lambda _{4}^{2}\;\psi _{4445}(u)]\varphi
_{0}^{2}(\ell )=0.
\end{equation*}%
The first part of the Proposition is hence proved. For the second part, we
can argue similarly and obtain%
\begin{align*}
k_{1}(u;\ell )& =k_{2}(u;\ell )=[\lambda _{3}\lambda _{5}\;\theta
_{35}(u)+\lambda _{2}\lambda _{3}\;\theta _{33}(u)-\lambda _{4}^{2}\;\theta
_{44}(u)]\varphi (0)\varphi (2) \\
& =-\frac{2\Phi (-u)+\lambda _{\ell }u\phi (u)}{4\pi },
\end{align*}

\begin{align*}
k_{3}(u;\ell )& =[\lambda _{3}\lambda _{5}\;\psi _{3335}(u)+\lambda
_{2}\lambda _{3}\;\psi _{3333}(u)-\lambda _{4}^{2}\;\psi _{3344}(u)]\varphi
^{2}(0) \\
& \;\;-[\lambda _{3}\lambda _{5}\;\theta _{35}(u)+\lambda _{2}\lambda
_{3}\;\theta _{33}(u)-\lambda _{4}^{2}\;\theta _{44}(u)]\varphi ^{2}(0) \\
& =\frac{\Phi (-u)(\lambda _{\ell }+4)(3\lambda _{\ell }-2)+8\lambda _{\ell
}(\lambda _{\ell }(u^{2}+1)+2)u\phi (u)}{4\pi (3\lambda _{\ell }-2)}, \\
& =\Phi (-u)\frac{\lambda _{\ell }+4}{4\pi }+\lambda _{\ell }\frac{\lambda
_{\ell }(2u^{2}+5)+2}{4\pi (3\lambda _{\ell }-2)}u\phi (u),
\end{align*}

\begin{align*}
k_{4}(u;\ell )& =[\lambda _{3}\lambda _{5}\;\psi _{3445}(u)+\lambda
_{2}\lambda _{3}\;\psi _{3344}(u)-\lambda _{4}^{2}\;\psi _{4444}(u)]\varphi
^{2}(0) \\
& \;\;-[\lambda _{3}\lambda _{5}\;\theta _{35}(u)+\lambda _{2}\lambda
_{3}\;\theta _{33}(u)-\lambda _{4}^{2}\;\theta _{44}(u)]\varphi ^{2}(0) \\
& =-\Phi (-u)\frac{\lambda _{\ell }-2}{4\pi },
\end{align*}%
and finally

\begin{align*}
k_{5}(u;\ell )& =[\lambda _{3}\lambda _{5}\;\psi _{3555}(u)+\lambda
_{2}\lambda _{3}\;\psi _{3355}(u)-\lambda _{4}^{2}\;\psi _{4455}(u)]\varphi
^{2}(0) \\
& \;\;-[\lambda _{3}\lambda _{5}\;\theta _{35}(u)+\lambda _{2}\lambda
_{3}\;\theta _{33}(u)-\lambda _{4}^{2}\;\theta _{44}(u)]\varphi ^{2}(0) \\
& =(\lambda _{\ell }-2)\frac{\lambda _{\ell }(u^{2}+1)+2}{4\pi (3\lambda
_{\ell }-2)}u\phi (u).
\end{align*}
\end{proof}

\section{Appendix B: Proof of Proposition \protect\ref{AB}}

\label{proofAB}

We need first to introduce some more notation concerning the integrals of
products of random eigenfunction and/or their derivatives. As before, we
denote by $e_{a}^{x}$, $a=1,2$, the covariant derivative at $x\in \mathbb{S}%
^{2}$ with respect to the first or second variable $\theta $, $\varphi $. We
have to deal with the following integrals of squares:%
\begin{equation}
I_{00}(\ell )=\int_{\mathbb{S}^{2}}f_{\ell }^{2}(x)dx,\text{ }I_{11}(\ell
)=\int_{\mathbb{S}^{2}}\left\{ e_{1}^{x}f_{\ell }(x)\right\} ^{2}dx,\text{ }%
I_{22}(\ell )=\int_{\mathbb{S}^{2}}\left\{ e_{2}^{x}f_{\ell }(x)\right\}
^{2}dx;  \label{3marzo}
\end{equation}%
we shall also study the cross-product integral%
\begin{equation*}
I_{0,22}(\ell )=\int_{\mathbb{S}^{2}}f_{\ell }(x)e_{2}^{x}e_{2}^{x}f_{\ell
}(x)dx,
\end{equation*}%
and finally we shall consider%
\begin{equation*}
I_{12,12}(\ell )=\int_{\mathbb{S}^{2}}\left\{ e_{1}^{x}e_{2}^{x}f_{\ell
}(x)\right\} ^{2}dx,\text{ }I_{22,22}(\ell )=\int_{\mathbb{S}%
^{2}}\{e_{2}^{x}e_{2}^{x}f_{\ell }(x)\}^{2}dx.
\end{equation*}%
Let us now show how the analysis of these 6 integrals will suffice for our
needs. First note that, since 
\begin{equation*}
Y_{5}(x)=\frac{1}{\lambda _{5}}\left\{ e_{2}^{x}e_{2}^{x}f_{\ell }(x)-\frac{%
\lambda _{2}}{\lambda _{3}}e_{1}^{x}e_{1}^{x}f_{\ell }(x)\right\}
\end{equation*}%
and 
\begin{equation*}
e_{1}^{x}e_{1}^{x}f_{\ell }(x)=-\lambda _{\ell }f_{\ell
}(x)-e_{2}^{x}e_{2}^{x}f_{\ell }(x);
\end{equation*}%
we have 
\begin{align*}
A_{35}& =\frac{1}{\lambda _{3}\lambda _{5}}\int_{\mathbb{S}%
^{2}}e_{1}^{x}e_{1}^{x}f_{\ell }(x)\left\{ e_{2}^{x}e_{2}^{x}f_{\ell }(x)-%
\frac{\lambda _{2}}{\lambda _{3}}e_{1}^{x}e_{1}^{x}f_{\ell }(x)\right\} dx \\
& =-\frac{\lambda _{\ell }}{\lambda _{3}\lambda _{5}}\left\{ 1+2\frac{%
\lambda _{2}}{\lambda _{3}}\right\} I_{0,22}(\ell )-\frac{\lambda _{\ell
}^{2}\lambda _{2}}{\lambda _{3}^{2}\lambda _{5}}I_{00}(\ell )-\frac{1}{%
\lambda _{3}\lambda _{5}}\left\{ 1+\frac{\lambda _{2}}{\lambda _{3}}\right\}
I_{22,22}(\ell ).
\end{align*}%
Likewise%
\begin{eqnarray*}
B_{1} &=&\int_{\mathbb{S}^{2}}H_{2}\left( \frac{e_{1}^{x}f_{\ell }(x)}{%
\lambda _{1}}\right) dx=\frac{1}{\lambda _{1}^{2}}I_{11}(\ell )-4\pi \\
B_{2} &=&\int_{\mathbb{S}^{2}}H_{2}\left( \frac{e_{2}^{x}f_{\ell }(x)}{%
\lambda _{1}}\right) dx=\frac{1}{\lambda _{1}^{2}}I_{22}(\ell )-4\pi ,
\end{eqnarray*}%
so that these terms only require the investigation of integrals in (\ref%
{3marzo}). Finally, for the remaining terms it suffices to note that%
\begin{align*}
B_{3}& =\int_{\mathbb{S}^{2}}H_{2}\left( \frac{e_{1}^{x}e_{1}^{x}f_{\ell }(x)%
}{\lambda _{3}}\right) dx=\frac{\lambda _{\ell }^{2}}{\lambda _{3}^{2}}%
I_{00}(\ell )+\frac{1}{\lambda _{3}^{2}}I_{22,22}(\ell )+\frac{2\lambda
_{\ell }}{\lambda _{3}^{2}}I_{0,22}(\ell )-4\pi \\
B_{4}& =\int_{\mathbb{S}^{2}}H_{2}\left( \frac{e_{1}^{x}e_{2}^{x}f_{\ell }(x)%
}{\lambda _{4}}\right) dx=\frac{1}{\lambda _{4}^{2}}I_{12,12}(\ell )-4\pi
\end{align*}%
and 
\begin{align*}
B_{5}& =\int_{\mathbb{S}^{2}}H_{2}\left( \frac{1}{\lambda _{5}}%
e_{2}^{x}e_{2}^{x}f_{\ell }(x)-\frac{\lambda _{2}}{\lambda _{3}\lambda _{5}}%
e_{1}^{x}e_{1}^{x}f_{\ell }(x)\right) dx \\
& =\frac{1}{\lambda _{5}^{2}}\big(1+\frac{\lambda _{2}}{\lambda _{3}}\big)%
^{2}I_{22,22}(\ell )+\frac{\lambda _{\ell }^{2}\lambda _{2}^{2}}{\lambda
_{3}^{2}\lambda _{5}^{2}}I_{00}(\ell )+2\frac{\lambda _{\ell }\lambda _{2}}{%
\lambda _{3}\lambda _{5}^{2}}\big(1+\frac{\lambda _{2}}{\lambda _{3}}\big)%
I_{0,22}(\ell )-4\pi .
\end{align*}%
A crucial step in our argument is the possibility to write these integrals
explicitly in terms of the spherical harmonic coefficients $\left\{ a_{\ell
m}\right\} $. This task is accomplished in the following Lemma.

\begin{lemma}
a) For the integrals of square terms, we have that 
\begin{equation*}
I_{00}(\ell )= \frac{1}{2\ell +1}a_{\ell 0}^{2}+\frac{2}{2\ell +1}%
\sum_{m>0}|a_{\ell m}|^{2}=\frac{1}{2\ell +1}\sum_{m=-\ell }^{\ell }|a_{\ell
m}|^{2};
\end{equation*}%
\begin{equation*}
I_{11}(\ell ) =a_{\ell 0}^{2}\frac{\lambda_{\ell} }{2\ell +1}%
+\sum_{m>0}|a_{\ell m}|^{2}\left\{ 2\frac{\lambda _{\ell }}{2\ell +1}
-m\right\} =\sum_{m=-\ell }^{\ell }|a_{\ell m}|^{2}\left\{ \frac{\lambda
_{\ell }}{2\ell +1}-\frac{m}{2}\right\},
\end{equation*}
and 
\begin{equation*}
I_{22}(\ell )=\frac{1}{2}\sum_{m=-\ell }^{\ell }|a_{\ell m}|^{2}m.
\end{equation*}
b) For the cross-product integral, we have that 
\begin{equation*}
I_{0,22}(\ell )=-a_{\ell 0}^{2}\frac{\ell }{2\ell +1} +\sum_{m>0}|a_{\ell
m}|^{2}(\frac{1}{2\ell +1}-m).
\end{equation*}
c) Finally for the remaining terms 
\begin{equation*}
I_{12,12}(\ell ) =\sum_{m>0}|a_{\ell m}|^{2}m\left\{ \frac{ \lambda -1-m^{2}%
}{2}\right\} =\sum_{m=-\ell }^{\ell }|a_{\ell m}|^{2}m\left\{ \frac{ \lambda
-1-m^{2}}{4}\right\},
\end{equation*}
and 
\begin{equation*}
I_{22,22}(\ell )=\frac{a_{\ell 0}^{2}}{2}\left( \ell ^{2}-\frac{\ell }{2\ell
+1}\right) +\frac{1}{2}\sum_{m>0}|a_{\ell m}|^{2}\left\{ -\frac{4\lambda }{
2\ell +1}+m+\lambda m+m^{3}\right\}.
\end{equation*}
\end{lemma}

\begin{proof}
We introduce here the standard basis for spherical harmonics, see, i.e., 
\cite{MaPeCUP}, Section 13.2, which is given by 
\begin{equation*}
Y_{\ell m}(\theta ,\varphi )=%
\begin{cases}
e^{im\varphi }\sqrt{\frac{2\ell +1}{4\pi }\frac{(\ell -m)!}{(\ell +m)!}}%
P_{\ell }^{m}(\cos \theta ), & m\geq 0, \\ 
(-1)^{m}e^{im\varphi }\sqrt{\frac{2\ell +1}{4\pi }\frac{(\ell +m)!}{(\ell
-m)!}}P_{\ell }^{-m}(\cos \theta ), & m<0,%
\end{cases}%
\end{equation*}%
where we introduced also the associated Legendre functions, which are
defined by 
\begin{equation*}
P_{\ell }^{m}(x)=%
\begin{cases}
(-1)^{m}(1-x^{2})^{m/2}\frac{d^{m}}{dx^{m}}P_{\ell }(x), & m\geq 0, \\ 
(-1)^{m}\frac{(\ell +m)!}{(\ell -m)!}P_{\ell }^{-m}(x), & m<0.%
\end{cases}%
\end{equation*}%
Let us recall also the trivial orthogonality relationships 
\begin{equation*}
\int_{0}^{2\pi }e^{im\varphi }e^{in\varphi }d\varphi =%
\begin{cases}
2\pi & n=-m, \\ 
0 & n\neq -m,%
\end{cases}%
\end{equation*}%
which yield 
\begin{equation*}
\frac{1}{2\ell +1}\int_{0}^{2\pi }Y_{\ell m}(\theta ,\varphi )Y_{\ell
n}(\theta ,\varphi )d\varphi =%
\begin{cases}
\frac{1}{2}\frac{(\ell -m)!}{(\ell +m)!}(-1)^{m}\{P_{\ell }^{m}(\cos \theta
)\}^{2} & n=-m, \\ 
0 & n\neq -m.%
\end{cases}%
\end{equation*}%
Our next tool are the analytic expression for derivatives of spherical
harmonics, which we recall to be given by 
\begin{equation*}
e_{1}^{x}Y_{\ell m}(x)=\frac{\partial }{\partial \theta }Y_{\ell m}(\theta
,\varphi ),\;\;\;\;e_{2}^{x}Y_{\ell m}(x)=\frac{1}{\sin \theta }\frac{%
\partial }{\partial \varphi }Y_{\ell m}(\theta ,\varphi )=\frac{im}{\sin
\theta }Y_{\ell m}(\theta ,\varphi ),\text{ }
\end{equation*}%
and moreover%
\begin{eqnarray*}
e_{1}^{x}e_{2}^{x}Y_{\ell m}(x) &=&\frac{1}{\sin \theta }\frac{\partial }{%
\partial \varphi }\frac{\partial }{\partial \theta }Y_{\ell m}(\theta
,\varphi )-\frac{\cos \theta }{\sin ^{2}\theta }\frac{\partial }{\partial
\varphi }Y_{\ell m}(\theta ,\varphi ) \\
&=&\frac{im}{\sin \theta }\frac{\partial }{\partial \theta }Y_{\ell
m}(\theta ,\varphi )-im\frac{\cos \theta }{\sin ^{2}\theta }Y_{\ell
m}(\theta ,\varphi ),
\end{eqnarray*}%
\begin{eqnarray*}
e_{2}^{x}e_{2}^{x}Y_{\ell m}(x) &=&\frac{1}{\sin ^{2}\theta }\frac{\partial
^{2}}{\partial \varphi ^{2}}Y_{\ell m}(\theta ,\varphi )+\frac{\cos \theta }{%
\sin \theta }\frac{\partial }{\partial \theta }Y_{\ell m}(\theta ,\varphi )
\\
&=&-\frac{m^{2}}{\sin ^{2}\theta }\;Y_{\ell m}(\theta ,\varphi )+\frac{\cos
\theta }{\sin \theta }\frac{\partial }{\partial \theta }Y_{\ell m}(\theta
,\varphi ).
\end{eqnarray*}%
Finally, we recall that the spherical harmonic coefficients satisfy the
following identities, see again \cite{MaPeCUP}, formula (6.6): 
\begin{equation*}
(-1)^{m}a_{\ell ,-m}=\bar{a}_{\ell m},\hspace{2cm}(-1)^{m}a_{\ell m}a_{\ell
,-m}=|a_{\ell m}|^{2}.
\end{equation*}%
The first part of a) is a trivial consequence of the Parseval's identity, or
the orthonormality of spherical harmonics:%
\begin{equation*}
I_{00}(\ell )=\int_{\mathbb{S}^{2}}f_{\ell }^{2}(x)dx=\frac{1}{2\ell +1}%
\sum_{m=-\ell }^{\ell }|a_{\ell m}|^{2}.
\end{equation*}%
For the other two integrals in a), the first step is to rewrite them as
functions of derivatives of associated Legendre functions, as follows:%
\begin{equation*}
I_{11}(\ell )=\int_{\mathbb{S}^{2}}\left\{ e_{1}^{x}f_{\ell }(x)\right\}
^{2}dx
\end{equation*}%
\begin{equation*}
=\frac{a_{\ell 0}^{2}}{2}\int_{0}^{\pi }\left\{ \frac{d}{d\theta }P_{\ell
}^{0}(\cos \theta )\right\} ^{2}\sin \theta d\theta +\sum_{m>0}|a_{\ell
m}|^{2}\frac{(\ell -m)!}{(\ell +m)!}\int_{0}^{\pi }\left\{ \frac{d}{d\theta }%
P_{\ell }^{m}(\cos \theta )\right\} ^{2}\sin \theta d\theta ,
\end{equation*}%
\begin{equation*}
I_{22}(\ell )=\int_{\mathbb{S}^{2}}\left\{ e_{2}^{x}f_{\ell }(x)\right\}
^{2}dx
\end{equation*}%
\begin{equation*}
=\sum_{m>0}|a_{\ell m}|^{2}m^{2}\frac{(\ell -m)!}{(\ell +m)!}\int_{0}^{\pi }%
\frac{1}{\sin ^{2}\theta }\left\{ P_{\ell }^{m}(\cos \theta )\right\}
^{2}\sin \theta d\theta .
\end{equation*}%
The same approach is needed to rewrite the integral in b):%
\begin{equation*}
I_{0,22}(\ell )=\int_{\mathbb{S}^{2}}f_{\ell }(x)e_{2}^{x}e_{2}^{x}f_{\ell
}(x)dx
\end{equation*}%
\begin{align*}
& =-\sum_{m>0}|a_{\ell m}|^{2}\frac{(\ell -m)!}{(\ell +m)!}%
m^{2}\int_{0}^{\pi }\frac{1}{\sin ^{2}\theta }\left\{ P_{\ell }^{m}(\cos
\theta )\right\} ^{2}\sin \theta d\theta \\
& \;\;+\frac{a_{\ell 0}^{2}}{2}\int_{0}^{\pi }\frac{\cos \theta }{\sin
\theta }P_{\ell }^{0}(\cos \theta )\frac{d}{d\theta }P_{\ell }^{0}(\cos
\theta )\sin \theta d\theta \\
& \;\;+\sum_{m>0}|a_{\ell m}|^{2}\frac{(\ell -m)!}{(\ell +m)!}\int_{0}^{\pi }%
\frac{\cos \theta }{\sin \theta }P_{\ell }^{m}(\cos \theta )\frac{d}{d\theta 
}P_{\ell }^{m}(\cos \theta )\sin \theta d\theta ,
\end{align*}%
and similarly for c):%
\begin{equation*}
I_{12,12}(\ell )=\int_{\mathbb{S}^{2}}\left\{ e_{1}^{x}e_{2}^{x}f_{\ell
}(x)\right\} ^{2}dx
\end{equation*}%
\begin{align*}
& =\sum_{m>0}|a_{\ell m}|^{2}\frac{(\ell -m)!}{(\ell +m)!}m^{2}\int_{0}^{\pi
}\frac{1}{\sin ^{2}\theta }\left\{ \frac{d}{d\theta }P_{\ell }^{m}(\cos
\theta )\right\} ^{2}\sin \theta d\theta \\
& \;\;-2\sum_{m>0}|a_{\ell m}|^{2}\frac{(\ell -m)!}{(\ell +m)!}%
m^{2}\int_{0}^{\pi }\frac{\cos \theta }{\sin ^{3}\theta }\left\{ \frac{d}{%
d\theta }P_{\ell }^{m}(\cos \theta )\right\} P_{\ell }^{m}(\cos \theta )\sin
\theta d\theta \\
& \;\;+\sum_{m>0}|a_{\ell m}|^{2}\frac{(\ell -m)!}{(\ell +m)!}%
m^{2}\int_{0}^{\pi }\frac{\cos ^{2}\theta }{\sin ^{4}\theta }\left\{ P_{\ell
}^{m}(\cos \theta )\right\} ^{2}\sin \theta d\theta ,
\end{align*}%
\begin{equation*}
I_{22,22}(\ell )=\int_{\mathbb{S}^{2}}\{e_{2}^{x}e_{2}^{x}f_{\ell
}(x)\}^{2}dx
\end{equation*}%
\begin{align*}
& =\sum_{m>0}|a_{\ell m}|^{2}\frac{(\ell -m)!}{(\ell +m)!}m^{4}\int_{0}^{\pi
}\frac{1}{\sin ^{4}\theta }\left\{ P_{\ell }^{m}(\cos \theta )\right\}
^{2}\sin \theta d\theta \\
& \;\;-2\sum_{m>0}|a_{\ell m}|^{2}\frac{(\ell -m)!}{(\ell +m)!}%
m^{2}\int_{0}^{\pi }\frac{\cos \theta }{\sin ^{3}\theta }P_{\ell }^{m}(\cos
\theta )\left\{ \frac{d}{d\theta }P_{\ell }^{m}(\cos \theta )\right\} \sin
\theta d\theta \\
& \;\;+\frac{a_{\ell 0}^{2}}{2}\int_{0}^{\pi }\frac{\cos ^{2}\theta }{\sin
^{2}\theta }\left\{ \frac{d}{d\theta }P_{\ell }^{0}(\cos \theta )\right\}
^{2}\sin \theta d\theta \\
& \;\;+\sum_{m>0}|a_{\ell m}|^{2}\frac{(\ell -m)!}{(\ell +m)!}\int_{0}^{\pi }%
\frac{\cos ^{2}\theta }{\sin ^{2}\theta }\left\{ \frac{d}{d\theta }P_{\ell
}^{m}(\cos \theta )\right\} ^{2}\sin \theta d\theta .
\end{align*}%
It is now convenient to introduce the following, more compact notation for
integrals of associated Legendre functions and their derivatives; more
precisely, we shall write%
\begin{equation*}
J_{1}(\ell ,m):=\frac{(\ell -m)!}{(\ell +m)!}\int_{0}^{\pi }\left\{ \frac{d}{%
d\theta }P_{\ell }^{m}(\cos \theta )\right\} ^{2}\sin \theta d\theta ,
\end{equation*}%
\begin{equation*}
J_{2}(\ell ,m):=\frac{(\ell -m)!}{(\ell +m)!}\int_{0}^{\pi }\frac{1}{\sin
^{2}\theta }\left\{ P_{\ell }^{m}(\cos \theta )\right\} ^{2}\sin \theta
d\theta ,
\end{equation*}%
\begin{equation*}
J_{3}(\ell ,m):=\frac{(\ell -m)!}{(\ell +m)!}\int_{0}^{\pi }\frac{\cos
\theta }{\sin \theta }P_{\ell }^{m}(\cos \theta )\frac{d}{d\theta }P_{\ell
}^{m}(\cos \theta )\sin \theta d\theta ,
\end{equation*}%
\begin{equation*}
J_{4}(\ell ,m):=\frac{(\ell -m)!}{(\ell +m)!}\int_{0}^{\pi }\frac{1}{\sin
^{2}\theta }\left\{ \frac{d}{d\theta }P_{\ell }^{m}(\cos \theta )\right\}
^{2}\sin \theta d\theta ,
\end{equation*}%
\begin{equation*}
J_{5}(\ell ,m):=\frac{(\ell -m)!}{(\ell +m)!}\int_{0}^{\pi }\frac{\cos
\theta }{\sin ^{3}\theta }\left\{ \frac{d}{d\theta }P_{\ell }^{m}(\cos
\theta )\right\} P_{\ell }^{m}(\cos \theta )\sin \theta d\theta ,
\end{equation*}%
\begin{equation*}
J_{6}(\ell ,m):=\frac{(\ell -m)!}{(\ell +m)!}\int_{0}^{\pi }\frac{\cos
^{2}\theta }{\sin ^{4}\theta }\left\{ P_{\ell }^{m}(\cos \theta )\right\}
^{2}\sin \theta d\theta ,
\end{equation*}%
\begin{equation*}
J_{7}(\ell ,m):=\frac{(\ell -m)!}{(\ell +m)!}\int_{0}^{\pi }\frac{1}{\sin
^{4}\theta }\left\{ P_{\ell }^{m}(\cos \theta )\right\} ^{2}\sin \theta
d\theta ,
\end{equation*}%
and 
\begin{equation*}
J_{8}(\ell ,m)=\frac{(\ell -m)!}{(\ell +m)!}\int_{0}^{\pi }\frac{\cos
^{2}\theta }{\sin ^{2}\theta }\left\{ \frac{d}{d\theta }P_{\ell }^{m}(\cos
\theta \right\} ^{2}\sin \theta d\theta .
\end{equation*}%
It is then readily verified that%
\begin{equation*}
I_{11}(\ell )=\frac{a_{\ell 0}^{2}}{2}J_{1}(\ell ,0)+\sum_{m>0}|a_{\ell
m}|^{2}J_{1}(\ell ,m),\;\;\;\;I_{22}(\ell )=\sum_{m>0}|a_{\ell
m}|^{2}m^{2}J_{2}(\ell ,m);
\end{equation*}%
moreover%
\begin{eqnarray*}
I_{0,22}(\ell ) &=&-\sum_{m>0}|a_{\ell m}|^{2}m^{2}J_{2}(\ell ,m)+\frac{%
a_{\ell 0}^{2}}{2}J_{3}(\ell ,0)+\sum_{m>0}|a_{\ell m}|^{2}J_{3}(\ell ,m) \\
&=&\frac{a_{\ell 0}^{2}}{2}J_{3}(\ell ,0)+\sum_{m>0}|a_{\ell
m}|^{2}\{J_{3}(\ell ,m)-m^{2}J_{2}(\ell ,m)\},
\end{eqnarray*}%
and%
\begin{eqnarray*}
I_{12,12}(\ell ) &=&\sum_{m>0}|a_{\ell m}|^{2}m^{2}J_{4}(\ell
,m)-2\sum_{m>0}|a_{\ell m}|^{2}m^{2}J_{5}(\ell ,m)+\sum_{m>0}|a_{\ell
m}|^{2}m^{2}J_{6}(\ell ,m) \\
&=&\sum_{m>0}|a_{\ell m}|^{2}m^{2}\{J_{4}(\ell ,m)-2J_{5}(\ell
,m)+J_{6}(\ell ,m)\},
\end{eqnarray*}%
\begin{equation*}
I_{22,22}(\ell )=\sum_{m>0}|a_{\ell m}|^{2}m^{4}J_{7}(\ell
,m)-2\sum_{m>0}|a_{\ell m}|^{2}m^{2}J_{5}(\ell ,m)+\frac{a_{\ell 0}^{2}}{2}%
J_{8}(\ell ,0)+\sum_{m>0}|a_{\ell m}|^{2}J_{8}(\ell ,m)
\end{equation*}%
\begin{equation*}
=\frac{a_{\ell 0}^{2}}{2}J_{8}(\ell ,0)+\sum_{m>0}|a_{\ell
m}|^{2}\{m^{4}J_{7}(\ell ,m)-2m^{2}J_{5}(\ell ,m)+J_{8}(\ell ,m)\}.
\end{equation*}%
The proof can then be completed by an explicit computation for the integrals 
$J_{a}(\ell ,m),$ $a=1,...,7,$ which is given in the following Lemma.
\end{proof}

\begin{lemma}
The following explicit evaluations hold for all $m=-\ell ,...,\ell$: 
\begin{equation*}
J_{1}(\ell ,m) =2\frac{\ell (\ell +1)}{2\ell +1}-m,
\end{equation*}
for $m\neq 0$ we have 
\begin{equation*}
J_{2}(\ell ,m)=\frac{1}{m}, \;\;\;\; J_{3}(\ell ,m)=\frac{1}{2\ell +1},
\end{equation*}
and, for $m \ne 0, \pm 1$, we also have 
\begin{equation*}
J_{4}(\ell ,m) =\frac{m}{2}\frac{\ell ^{2}+\ell +1-m^{2}}{m^{2}-1}, \;\;
J_{8}(\ell ,m)=\frac{1}{2}\left\{ m+\frac{\ell (\ell +1)(4+m+2\ell m-4m^{2})%
}{(2\ell +1)(m^{2}-1)}\right\} ,
\end{equation*}
\begin{equation*}
J_{5}(\ell ,m)=\frac{\ell (\ell +1)}{2m(m^{2}-1)}, \;\; J_{6}(\ell ,m)=\frac{%
\ell ^{2}+\ell +1-m^{2}}{2m(m^{2}-1)}, \;\; J_{7}(\ell ,m)=\frac{\ell
^{2}+\ell -1+m^{2}}{2m(m^{2}-1)}.
\end{equation*}
In particular we note that, for all $m \ne 0$, the following identities hold 
\begin{equation*}
J_{3}(\ell ,m)-m^{2}J_{2}(\ell ,m)=\frac{1}{2\ell +1}-m,
\end{equation*}%
\begin{equation*}
J_{4}(\ell ,m)-2J_{5}(\ell ,m)+J_{6}(\ell ,m)=\frac{\ell (\ell +1)-m^{2}-1}{%
2m},
\end{equation*}%
\begin{equation*}
m^{4}J_{7}(\ell ,m)-2m^{2}J_{5}(\ell ,m)+J_{8}(\ell ,m)=\frac{1}{2}\left\{ -4%
\frac{\ell (\ell +1)}{2\ell +1}+m+\ell (\ell +1)m+m^{3}\right\}.
\end{equation*}
and that, for $m=0$, we have 
\begin{equation*}
J_{1}(\ell ,0) =2\frac{\ell (\ell +1)}{2\ell +1}, \;\;\;\; J_{3}(\ell ,0) =-%
\frac{2 \ell }{2\ell +1},\;\;\;\; J_{8}(\ell ,0) =\ell^2 - \frac{\ell}{2
\ell+1}.
\end{equation*}
\end{lemma}

\begin{proof}
The proofs are all easy consequences of some simple change of variables
formulae and the analytic results on integrals of Associated Legendre
Functions which we collected in Section \ref{IntALF}. More precisely,
exploiting Lemma \ref{ALFsquares} one obtains: 
\begin{equation*}
J_{2}(\ell ,m)=\frac{(\ell -m)!}{(\ell +m)!}\int_{-1}^{1}\frac{1}{1-z^{2}}%
\left\{ P_{\ell }^{m}(z)\right\} ^{2}dz=\frac{1}{m},
\end{equation*}%
in view of (\ref{111m}) and (\ref{diffm}), moreover, by applying (\ref{diffm}%
), we have 
\begin{equation*}
J_{6}(\ell ,m)=\frac{(\ell -m)!}{(\ell +m)!}\int_{-1}^{1}\frac{z^{2}}{%
(1-z^{2})^{2}}\left\{ P_{\ell }^{m}(z)\right\} ^{2}dz=\frac{\ell ^{2}+\ell
+1-m^{2}}{2m(m^{2}-1)},
\end{equation*}%
and from (\ref{111m}) 
\begin{equation*}
J_{7}(\ell ,m)=\frac{(\ell -m)!}{(\ell +m)!}\int_{-1}^{1}\frac{1}{%
(1-z^{2})^{2}}\left\{ P_{\ell }^{m}(z)\right\} ^{2}dz=\frac{\ell ^{2}+\ell
-1+m^{2}}{2m(m^{2}-1)}.
\end{equation*}%
Similarly, from (\ref{samaddar_2m}) we have 
\begin{equation*}
J_{3}(\ell ,m)=-\frac{(\ell -m)!}{(\ell +m)!}\int_{-1}^{1}zP_{\ell }^{m}(z)%
\frac{d}{dz}P_{\ell }^{m}(z)dz=\frac{1}{2\ell +1},
\end{equation*}%
and, in view of Lemma \ref{ALFcrossproducts}, 
\begin{equation*}
J_{5}(\ell ,m)=-\frac{(\ell -m)!}{(\ell +m)!}\int_{-1}^{1}\frac{z}{1-z^{2}}%
P_{\ell }^{m}(z)\left\{ \frac{d}{dz}P_{\ell }^{m}(z)\right\} dz=\frac{\ell
(\ell +1)}{2m(m^{2}-1)}.
\end{equation*}%
Finally, using Lemma \ref{ALFsquaredderivatives}, from (\ref{444m}) we have 
\begin{equation*}
J_{4}(\ell ,m)=\frac{(\ell -m)!}{(\ell +m)!}\int_{-1}^{1}\left\{ \frac{d}{dz}%
P_{\ell }^{m}(z)\right\} ^{2}dz=\frac{m}{2}\frac{\ell ^{2}+\ell +1-m^{2}}{%
m^{2}-1};
\end{equation*}%
from (\ref{333m}) we have 
\begin{equation*}
J_{1}(\ell ,m)=\frac{(\ell -m)!}{(\ell +m)!}\int_{-1}^{1}(1-z^{2})\left\{ 
\frac{d}{dz}P_{\ell }^{m}(z)\right\} ^{2}dz=2\frac{\ell (\ell +1)}{2\ell +1}%
-m,
\end{equation*}%
and, in view of (\ref{J8m}), 
\begin{equation*}
J_{8}(\ell ,m)=\frac{(\ell -m)!}{(\ell +m)!}\int_{-1}^{1}z^{2}\left\{ \frac{d%
}{dz}P_{\ell }^{m}(z)\right\} ^{2}dz=\frac{1}{2}\left\{ m+\frac{\ell (\ell
+1)(4+m+2\ell m-4m^{2})}{(2\ell +1)(m^{2}-1)}\right\} .
\end{equation*}
\end{proof}

\subsection{Appendix C: Some Integrals of Associated Legendre Functions 
\label{IntALF}}

In this final Appendix, we need to report some explicit computations on
integrals involving cross products of associated Legendre functions and
their derivatives. For some of these results we managed to find references,
others may be known already but we failed to locate any suitable reference
and therefore we report their proofs entirely; we believe they may have some
independent interest for related works on the geometry of random spherical \
harmonics. In particular, the following two results are given in \cite%
{samaddar} equation (25) and equation (37), respectively 
\begin{equation}
\int_{-1}^{1}\frac{1}{1-z^{2}}\left\{ P_{\ell }^{m}(z)\right\} ^{2}dz=\frac{%
(\ell +m)!}{m(\ell -m)!},  \label{samaddar_1m}
\end{equation}%
\begin{equation}
\int_{-1}^{1}zP_{\ell }^{m}(z)\left\{ \frac{d}{dz}P_{\ell }^{m}(z)\right\}
dz=\delta _{0,m}-\frac{(\ell +m)!}{(2\ell +1)(\ell -m)!}.
\label{samaddar_2m}
\end{equation}%
The other integrals we shall need are given in the following three Lemmas;
the first deals with squares of associated Legendre functions, the second
with cross-product of Legendre functions and their derivatives, the third
with squared derivatives.

\begin{lemma}
\label{ALFsquares} The following analytic expressions hold for all values of 
$\ell =1,2,3, \dots$: 
\begin{align}
&\int_{-1}^{1}\frac{1}{(1-z^{2})^{2}}\left\{ P_{\ell }^{m}(z)\right\} ^{2}dz
\label{111m} \\
&=\frac{1}{4m^{2}}\left\{ \frac{(\ell +m)(\ell +m-1)(\ell +m)!}{(m-1)(\ell
-m)!}+\frac{(\ell +m)!}{(m+1)(\ell -m-2)!}\right\},  \notag
\end{align}
\begin{align}
&\int_{-1}^{1}\frac{z^{2}}{(1-z^{2})^{2}}\left\{ P_{\ell }^{m}(z)\right\}
^{2}dz  \label{diffm} \\
&=\frac{1}{4m^{2}}\left\{ \frac{(\ell +m)(\ell +m-1)(\ell +m)!}{(m-1)(\ell
-m)!}+\frac{(\ell +m)!}{(m+1)(\ell -m-2)!}\right\} -\frac{(\ell +m)!}{m(\ell
-m)!}.  \notag
\end{align}
\end{lemma}

\begin{proof}
\noindent \noindent Formula (\ref{diffm}) follows from (\ref{111m}) and (\ref%
{samaddar_1m}): 
\begin{equation*}
\int_{-1}^{1}\frac{z^{2}}{(1-z^{2})^{2}}(P_{\ell
}^{m}(z))^{2}dz=\int_{-1}^{1}\frac{1}{(1-z^{2})^{2}}(P_{\ell
}^{m}(z))^{2}dz-\int_{-1}^{1}\frac{1}{1-z^{2}}(P_{\ell }^{m}(z))^{2}dz.
\end{equation*}%
To prove (\ref{111m}) we exploit the following identity (see i.e., \cite%
{lebedev}, Section 7.12): 
\begin{equation*}
\frac{1}{\sqrt{1-z^{2}}}P_{\ell }^{m}(z)=-\frac{1}{2m}\left[ (\ell
+m-1)(\ell +m)P_{\ell -1}^{m-1}(z)+P_{\ell -1}^{m+1}(z)\right]
\end{equation*}%
whence%
\begin{equation*}
\int_{-1}^{1}\frac{1}{(1-z^{2})^{2}}\left\{ P_{\ell }^{m}(z)\right\}
^{2}dz=\int_{-1}^{1}\frac{1}{1-z^{2}}\left\{ \frac{P_{\ell }^{m}(z)}{\sqrt{%
1-z^{2}}}\right\} ^{2}dz
\end{equation*}%
\begin{equation*}
=\frac{1}{4m^{2}}\int_{-1}^{1}\frac{1}{1-z^{2}}\{(\ell +m-1)(\ell +m)P_{\ell
-1}^{m-1}(z)+P_{\ell -1}^{m+1}(z)\}^{2}dz
\end{equation*}%
\begin{align*}
& =\frac{(\ell +m-1)^{2}(\ell +m)^{2}}{4m^{2}}\int_{-1}^{1}\frac{1}{1-z^{2}}%
\{P_{\ell -1}^{m-1}(z)\}^{2}dz \\
& \;\;+\frac{(\ell +m-1)(\ell +m)}{2m^{2}}\int_{-1}^{1}\frac{1}{1-z^{2}}%
P_{\ell -1}^{m-1}(z)P_{\ell -1}^{m+1}(z)dz+\frac{1}{4m^{2}}\int_{-1}^{1}%
\frac{1}{1-z^{2}}\{P_{\ell -1}^{m+1}(z)\}^{2}dz;
\end{align*}%
the statement immediately follows by applying twice equation (\ref%
{samaddar_1m}): 
\begin{align*}
& \int_{-1}^{1}\frac{1}{1-z^{2}}\left\{ P_{\ell -1}^{m-1}(z)\right\} ^{2}dz=%
\frac{(\ell +m-2)!}{(m-1)(\ell -m)!},\int_{-1}^{1}\frac{1}{1-z^{2}}\left\{
P_{\ell -1}^{m+1}(z)\right\} ^{2}dz \\
& =\frac{(\ell +m)!}{(m+1)(\ell -m-2)!},
\end{align*}%
and by observing that 
\begin{equation*}
\int_{-1}^{1}\frac{1}{1-z^{2}}P_{\ell -1}^{m-1}(z)P_{\ell -1}^{m+1}(z)dz=0.
\end{equation*}%
\noindent
\end{proof}

\begin{lemma}
\label{ALFcrossproducts}The following analytic expressions hold for all
values of $\ell =1,2,3,\dots$: 
\begin{align}
&\int_{-1}^{1}\frac{z}{1-z^{2}}P_{\ell }^{m}(z)\left\{ \frac{d}{dz}P_{\ell
}^{m}(z)\right\} dz  \notag \\
&=\frac{1}{4m}\left\{ \frac{(\ell +m+1)!}{(m+1)(\ell -m-1)!}-\frac{(\ell
+m)(\ell -m+1)(\ell +m)!}{(m-1)(\ell -m)!}\right\}.  \label{555m}
\end{align}
\end{lemma}

\begin{proof}
We first note that 
\begin{align*}
& zP_{\ell }^{m}(z)=-\frac{\sqrt{1-z^{2}}}{2m}\left[ (\ell +m)(\ell
-m+1)P_{\ell }^{m-1}(z)+P_{\ell }^{m+1}(z)\right] , \\
& \sqrt{1-z^{2}}\frac{d}{dz}P_{\ell }^{m}(z)=\frac{1}{2}\left[ (\ell
+m)(\ell -m+1)P_{\ell }^{m-1}(z)-P_{\ell }^{m+1}(z)\right] ,
\end{align*}%
so that%
\begin{equation*}
\int_{-1}^{1}\frac{z}{1-z^{2}}P_{\ell }^{m}(z)\left\{ \frac{d}{dz}P_{\ell
}^{m}(z)\right\} dz
\end{equation*}%
\begin{equation*}
=-\frac{1}{4m}\int_{-1}^{1}\frac{1}{1-z^{2}}\left[ (\ell +m)^{2}(\ell
-m+1)^{2}\{P_{\ell }^{m-1}(z)\}^{2}-\{P_{\ell }^{m+1}(z)\}^{2}\right] dz
\end{equation*}%
\begin{equation*}
=-\frac{(\ell +m)^{2}(\ell -m+1)^{2}}{4m}\int_{-1}^{1}\frac{1}{1-z^{2}}%
\left\{ P_{\ell }^{m-1}(z)\right\} ^{2}dz+\frac{1}{4m}\int_{-1}^{1}\frac{1}{%
1-z^{2}}\left\{ P_{\ell }^{m+1}(z)\right\} ^{2}dz
\end{equation*}%
and, by applying (\ref{samaddar_1m}), we immediately have the statement.
\end{proof}

\begin{lemma}
\label{ALFsquaredderivatives}The following analytic expressions hold for all
values of $\ell =1,2,3, \dots$: 
\begin{equation}
\int_{-1}^{1}\left\{ \frac{d}{dz}P_{\ell }^{m}(z)\right\} ^{2}dz=\frac{1}{4}%
\left\{ \frac{(\ell +m)(\ell -m+1)(\ell +m)!}{(m-1)(\ell -m)!}+\frac{(\ell
+m+1)!}{(m+1)(\ell -m-1)!}\right\},  \label{444m}
\end{equation}
\begin{align}
\int_{-1}^{1}(1-z^{2})\left\{ \frac{d}{dz}P_{\ell }^{m}(z)\right\} ^{2}dz &=%
\frac{1}{(2\ell +1)^{2}}\Big\{ \frac{(\ell +1)^{2}(\ell +m)(\ell +m)!}{%
m(\ell -m-1)!}-2\frac{\ell (\ell +1)(\ell -m+1)(\ell +m)!}{m(\ell -m-1)!} 
\notag \\
&+\frac{\ell ^{2}(\ell -m+1)^{2}(\ell +m+1)!}{m(\ell -m+1)!}\Big\},
\label{333m}
\end{align}%
\begin{align}
\int_{-1}^{1}z^{2}\left\{ \frac{d}{dz}P_{\ell }^{m}(z)\right\} ^{2}dz &=-%
\frac{1}{(2\ell +1)^{2}}\Big\{ \frac{(\ell +1)^{2}(\ell +m)(\ell +m)!}{
m(\ell -m-1)!}-2\frac{\ell (\ell +1)(\ell -m+1)(\ell +m)!}{m(\ell -m-1)!} 
\notag \\
&+ \frac{\ell ^{2}(\ell -m+1)^{2}(\ell +m+1)!}{m(\ell -m+1)!}\Big\} +\frac{1%
}{4}\Big\{ \frac{(\ell +m)(\ell -m+1)(\ell +m)!}{(m-1)(\ell -m)!}  \notag \\
&+\frac{(\ell +m+1)!}{(m+1)(\ell -m-1)!}\Big\} .  \label{J8m}
\end{align}
\end{lemma}

\begin{proof}
To prove (\ref{444m}) we use 
\begin{equation*}
\frac{d}{dz}P_{\ell }^{m}(z)=\frac{1}{2\sqrt{1-z^{2}}}\left\{ (\ell +m)(\ell
-m+1)P_{\ell }^{m-1}(z)-P_{\ell }^{m+1}(z)\right\}
\end{equation*}%
so that we may write 
\begin{equation*}
\int_{-1}^{1}\left\{ \frac{d}{dz}P_{\ell }^{m}(z)\right\} ^{2}dz=\frac{1}{4}%
\int_{-1}^{1}\frac{1}{1-z^{2}}\left\{ (\ell +m)(\ell -m+1)P_{\ell
}^{m-1}(z)-P_{\ell }^{m+1}(z)\right\} ^{2}dz
\end{equation*}%
\begin{equation*}
=\frac{(\ell +m)^{2}(\ell -m+1)^{2}}{4}\int_{-1}^{1}\frac{1}{1-z^{2}}%
\{P_{\ell }^{m-1}(z)\}^{2}dz
\end{equation*}%
\begin{equation*}
+\frac{1}{4}\int_{-1}^{1}\frac{1}{1-z^{2}}\left\{ P_{\ell }^{m+1}(z)\right\}
^{2}dz-\frac{(\ell +m)(\ell -m+1)}{2}\int_{-1}^{1}\frac{1}{1-z^{2}}P_{\ell
}^{m-1}(z)P_{\ell }^{m+1}(z)dz.
\end{equation*}%
Formula (\ref{444m}) then follows by observing that, from (\ref{samaddar_1m}%
), we have 
\begin{equation*}
\frac{(\ell +m)^{2}(\ell -m+1)^{2}}{4}\int_{-1}^{1}\frac{1}{1-z^{2}}\left\{
P_{\ell }^{m-1}(z)\right\} ^{2}dz=\frac{(\ell +m)^{2}(\ell -m+1)^{2}}{4}%
\frac{(\ell +m-1)!}{(m-1)(\ell -m+1)!},
\end{equation*}%
and 
\begin{equation*}
\frac{1}{4}\int_{-1}^{1}\frac{1}{1-z^{2}}\left\{ P_{\ell }^{m+1}(z)\right\}
^{2}dz=\frac{1}{4}\frac{(\ell +m+1)!}{(m+1)(\ell -m-1)!};
\end{equation*}%
and moreover 
\begin{equation*}
\int_{-1}^{1}\frac{1}{1-z^{2}}P_{\ell }^{m-1}(z)P_{\ell }^{m+1}(z)dz=0.
\end{equation*}%
To prove (\ref{333m}), we apply the following identity, see \cite{lebedev},
Section 7.12: 
\begin{equation*}
(1-z^{2})\frac{d}{dz}P_{\ell }^{m}(z)=\frac{1}{2\ell +1}\left\{ (\ell
+1)(\ell +m)P_{\ell -1}^{m}(z)-\ell (\ell -m+1)P_{\ell +1}^{m}(z)\right\}
\end{equation*}%
from which we obtain%
\begin{equation*}
\int_{-1}^{1}(1-z^{2})\left\{ \frac{d}{dz}P_{\ell }^{m}(z)\right\} ^{2}dz
\end{equation*}%
\begin{equation*}
=\frac{1}{(2\ell +1)^{2}}\int_{-1}^{1}\frac{1}{1-z^{2}}\{(\ell +1)(\ell
+m)P_{\ell -1}^{m}(z)-\ell (\ell -m+1)P_{\ell +1}^{m}(z)\}^{2}dz
\end{equation*}%
\begin{align*}
& =\frac{(\ell +1)^{2}(\ell +m)^{2}}{(2\ell +1)^{2}}\int_{-1}^{1}\frac{1}{%
1-z^{2}}\{P_{\ell -1}^{m}(z)\}^{2}dz \\
& \;\;-2\frac{(\ell +1)(\ell +m)\ell (\ell -m+1)}{(2\ell +1)^{2}}%
\int_{-1}^{1}\frac{1}{1-z^{2}}P_{\ell -1}^{m}(z)P_{\ell +1}^{m}(z)dz \\
& \;\;+\frac{\ell ^{2}(\ell -m+1)^{2}}{(2\ell +1)^{2}}\int_{-1}^{1}\frac{1}{%
1-z^{2}}\left\{ P_{\ell +1}^{m}(z)\right\} ^{2}dz.
\end{align*}%
Formula (\ref{333m}) follows by applying again (\ref{samaddar_1m}), which
gives 
\begin{equation*}
\int_{-1}^{1}\frac{1}{1-z^{2}}\left\{ P_{\ell -1}^{m}(z)\right\} ^{2}dz=%
\frac{(\ell +m-1)!}{m(\ell -m-1)!},\int_{-1}^{1}\frac{1}{1-z^{2}}\left\{
P_{\ell +1}^{m}(z)\right\} ^{2}=\frac{(\ell +m+1)!}{m(\ell -m+1)!},
\end{equation*}%
and \cite{samaddar}, formula (24i), which gives 
\begin{equation*}
\int_{-1}^{1}\frac{1}{1-z^{2}}P_{\ell -1}^{m}(z)P_{\ell +1}^{m}(z)dz=\frac{%
(\ell +m-1)!}{m(\ell -m-1)!}.
\end{equation*}%
\noindent Finally, to prove (\ref{J8m}) it is sufficient to note that%
\begin{equation*}
\int_{-1}^{1}z^{2}\left\{ \frac{d}{dz}P_{\ell }^{m}(z)\right\}
^{2}dz=-\int_{-1}^{1}(1-z^{2})\left\{ \frac{d}{dz}P_{\ell }^{m}(z)\right\}
^{2}dz+\int_{-1}^{1}\left\{ \frac{d}{dz}P_{\ell }^{m}(z)\right\} ^{2}dz
\end{equation*}%
\begin{equation*}
=-\frac{1}{(2\ell +1)^{2}}\Big\{\frac{(\ell +1)^{2}(\ell +m)(\ell +m)!}{%
m(\ell -m-1)!}-2\frac{\ell (\ell +1)(\ell -m+1)(\ell +m)!}{m(\ell -m-1)!}
\end{equation*}%
\begin{equation*}
+\frac{\ell ^{2}(\ell -m+1)^{2}(\ell +m+1)!}{m(\ell -m+1)!}\Big\}+\frac{1}{4}%
\Big\{\frac{(\ell +m)(\ell -m+1)(\ell +m)!}{(m-1)(\ell -m)!}+\frac{(\ell
+m+1)!}{(m+1)(\ell -m-1)!}\Big\}.
\end{equation*}
\end{proof}

\end{document}